\documentclass[11pt,a4paper]{amsart}
\usepackage{amssymb,amsmath,amsthm,enumerate,graphicx,mathtools,booktabs,pinlabel,float,tikz-cd}
\usepackage{stmaryrd,lscape,caption,subcaption,amsbsy,url}
\usepackage{tikz}
\usetikzlibrary{shapes,shadows,arrows,trees}
\usepackage{epstopdf}
\usepackage{inputenc}
\newtheorem{theorem}{Theorem}[section]
\newtheorem{prop}[theorem]{Proposition}
\newtheorem{prop*}{Proposition}
\newtheorem{lemma}[theorem]{Lemma}
\newtheorem{algo}[theorem]{Algorithm}

\newtheorem{cor*}{Corollary}
\theoremstyle{definition}
\newtheorem{defn}[theorem]{Definition}
\newtheorem{rem}[theorem]{Remark}
\newtheorem{cons}[theorem]{Construction}
\newtheorem{exmp}[theorem]{Example}
\newtheorem{exmp*}{Example}
\newcommand{\m}{\mathrm{Mod}(S_g)}
\newcommand{\Mod}{\mathrm{Mod}}
\newcommand{\Homeo}{\mathrm{Homeo}}

\newcommand{\lcm}{\mathrm{lcm}}

\renewcommand{\P}{\mathcal{P}}
\newcommand{\N}{\mathcal{N}}
\newcommand{\M}{\mathcal{M}}
\newcommand{\TM}{\widetilde{\mathcal{M}}}
\newcommand{\B}{\mathcal{B}}
\renewcommand{\O}{\mathcal{O}}
\newcommand{\C}{\mathcal{C}}
\newcommand{\D}{\mathcal{D}}
\newcommand{\F}{\mathcal{F}}
\newcommand{\T}{\mathcal{T}}
\newcommand{\W}{\mathcal{W}}
\newcommand{\Z}{\mathbb{Z}}

\renewcommand{\L}{\mathcal{L}}

\begin{document}
\title[Factoring periodic maps into Dehn twists]{Factoring periodic maps into Dehn twists}

\author[N. K. Dhanwani]{Neeraj K. Dhanwani}
\address{Department of Mathematics\\
Indian Institute of Science Education and Research Bhopal\\
Bhopal Bypass Road, Bhauri \\
Bhopal 462 066, Madhya Pradesh\\
India}
\email{neerajk.dhanwani@gmail.com}

\author[Ajay K. Nair]{Ajay K. Nair}
\address{Department of Mathematics\\
Indian Institute of Science\\
Bangalore - 560 012, Karnataka\\
India.}
\email{anir71@gmail.com}

\author[K. Rajeevsarathy]{Kashyap Rajeevsarathy}
\address{Department of Mathematics\\
Indian Institute of Science Education and Research Bhopal\\
Bhopal Bypass Road, Bhauri \\
Bhopal 462 066, Madhya Pradesh\\
India}
\email{kashyap@iiserb.ac.in}
\urladdr{https://home.iiserb.ac.in/$_{\widetilde{\phantom{n}}}$kashyap/}



\maketitle

\begin{abstract}
Let $\Mod(S_g)$ be the mapping class group of the closed orientable surface $S_g$ of genus $g \geq 1$. In this paper, we develop various methods for factoring periodic mapping classes into Dehn twists, up to conjugacy. As applications, we develop methods for factoring certain roots of Dehn twists as words in Dehn twists. We will also show the existence of conjugates of periodic maps of order $4g$ and $4g+2$, for $g\geq 2$, whose product is pseudo-Anosov.
\end{abstract}

\section{Introduction}

Let $S=S_{g,p}^b$ denote the orientable surface of genus $g \geq 0$ with $p \geq 0$ punctures and $b \geq 0$ boundary components, where we ignore the parameters $b$ or $p$ when they take the value zero. Let $\Mod(S)$ denote the mapping class group of $S$. Lickorish~\cite{WBRL} proved that $\m$ is generated by $3g-1$ Dehn twists about non-separating curves, and subsequently Humphries~\cite{SPH} showed that $\Mod(S_g)$ is generated by a minimal generating set comprising $2g+1$ Dehn twists about non-separating curves. Thus, it is a natural question to ask whether we can derive methods for representing an arbitrary periodic $F \in \Mod(S_g)$ as a word $\W(F)$ in Dehn twists, up to conjugacy. In this paper, we develop algorithms to write a word $\W(F)$ for arbitrary periodic mapping class $F\in\m$ in Dehn twists.

In their seminal paper~\cite{BH3}, Birman-Hilden derived an expression for $\W(F)$ when $F$ is of (largest possible) order $4g+2$ in $\Mod(S_g)$, and consequently for $F^{2g+1}$, the hyperelliptic involution. This problem was solved for involutions in $\Mod(S_2)$ by Matsumoto~\cite{MM}, which was later generalized to $g >2$ by Korkmaz~\cite{MK}. Using techniques in algebraic geometry, Hirose~\cite{SH} derived expressions for $\W(F)$ for every periodic $F \in \Mod(S_g)$, for $1 \leq g \leq 4$. However, these  specialized techniques of Hirose do not easily generalize for $g \geq 5$. In~\cite{PKS}, a method was described to decompose an arbitrary periodic mapping class $F \in \Mod(S_g)$ into irreducible components that are realized as rotations of certain canonical hyperbolic polygons with side-pairings. We use this decomposition to develop various methods in this paper for deriving $\W(F)$, for an arbitrary periodic mapping class $F \in \Mod(S_g)$. 

In Section~\ref{sec:rotation_word}, we provide a method for deriving $\W(F)$ when $F$ is realizable as a rotation of $S_g$. To begin with, we show that any non-free involution on $S_g$, for $g \geq 3$, can be decomposed into components, where each component is either a hyperelliptic involution on the torus or an involution in $S_2$ with two fixed points (We call these the \textit{fundamental involutions}). Using this decomposition and the Burkhardt \textit{handle swap} map~\cite{HB,MS1} (that swaps a pair of handles on $S_g$), we provide a method for factoring $F$ (into Dehn twists) when $F$ is non-free involution. Furthermore, we use the fact that any other surface rotation can be written as a product of at most two nonfree involutions, to factor all surface rotations (see Algorithm~\ref{algo:surf_rotn}). 

In Section~\ref{sec:chain_method}, we use the well known chain relation in $\Mod(S_g)$ to develop a method (that we call the \textit{chain method}) see~\ref{algo:chain_method}, for representing a large family of periodic mapping classes (that we will call \textit{chain-realizable} periodics) as words in Dehn twists. As an immediate application of the chain method, we represent the torsion elements in $\Mod(S_2)$ (up to conjugacy) as words in Dehn twists. In Section~\ref{sec:gen_star_method}, we apply a known generalization~\cite{NS,M00} of the standard star relation in $\Mod(S_1^3)$ to $g \geq 2$, to develop a method (that we call the \textit{star method}) for representing a even larger family of periodic mapping classes (that encompasses \textit{chain-realizable} periodics) that we call \textit{star-realizable} periodics, see algorithm~\ref{algo:star_methodFT}, as words in Dehn twists. For an $F \in \m$ that is neither rotational nor star-realizable, in Section~\ref{sec:symp_method}, we formulate a method for deriving $\W(F)$ that we call the \textit{symplectic method}, see algorithm~\ref{algo:symp_method}. As the name suggests, this method uses the symplectic representations of periodic maps as described in~\cite{PKS}.
 
  In Section~\ref{sec:mthd_appls}, we provide several applications of our methods. By applying the star and symplectic methods, we obtain representations for the torsion elements in $\Mod(S_3)$ (up to conjugacy) as words in Dehn twists.  When $c$ is a nonseparating curve in $S_g$, for $g \geq 2$, Margalit-Schleimer~\cite{MS} gave the first example of a nontrivial root of $T_c$ of degree $2g-1$. Furthermore, they derived an explicit factorization of this root into Dehn twists using the chain relation in $\Mod(S_g)$. By applying our star method and the theory developed in~\cite{MK1}, we provide an algorithm for factoring a root of $T_c$ of degree $2g-1$ that is non-conjugate to the Margalit-Schleimer root for $g \geq 3$. Consider the Lickorish generating set $\L_g = \{T_{a_1},T_{b_1},T_{c_1},T_{b_2},T_{a_2},T_{c_2}, \ldots,T_{c_{g-1}},T_{b_g},T_{a_g}\}$ for $\Mod(S_g)$ (where the curves $a_i$, $b_i$ and the $c_i$ are as indicated in Figure~\ref{fig:2_lick_gens}). In the same spirit, we apply results in~\cite{KR2} to derive a root of $T_{a_1}^{2g-2}$ of degree $4g-4$. Thus, we have the following.
 \begin{cor*}
For $g \geq 2$, consider the nonseparating curve $a_1$ in $S_g$.
\begin{enumerate}[(i)]
\item There exists a root of $T_{a_1}$ in $\Mod(S_g)$ of degree $2g-1$ which is nonconjugate to the Margalit-Schleimer root for $g \geq 3$ factoring into Dehn twists as $$T_{a_1}^{-1}(T_{c_1}T_{a_2}\prod_{i=2}^{g-1} (T_{b_i}T_{c_i})T_{b_g}T_{a_g})^2.$$
\item There exists a root of $T_{a_1}^{2g-2}$ of degree $4g-4$ in $\Mod(S_g)$ that factors into Dehn twists as 
$$T_{a_2}\prod_{i=1}^{g-1}(T_{c_i}T_{b_{i+1}}).$$
\end{enumerate}
\end{cor*}
\noindent When $c$ is separating curve in $S_g$ for $g \geq 2$, so that $S_g=S_{g_1}\#_cS_{g_2}$. It follows from the results in~\cite{KR1} that there exists a root of $T_c$ of degree $4g_2+2$. In this context, we  show that:
\begin{cor*}
For $g \geq 2$,  let $c$ be a separating curve in $S_g$ so that $S_g=S_{g_1}\#_cS_{g_2}$. Then there exists a root of $T_c$ in $\Mod(S_g)$ of degree $4g_2+2$ that factors into Dehn twists as
$$(\prod_{j=1}^{g_1}(T_{a_j}T_{b_j}T_{a_j'})^{2(-1)^{(g_1-j)}})(T_{a_{g_1+1}}^2(\prod_{i=g_1+1}^{g-1}T_{b_i}T_{c_i})T_{b_g}T_{a_g})^{g_2+1}T_{c}^{-1}.$$
\end{cor*}

\noindent Furthermore, we have provided a method for factoring the root of multitwist $\prod_{i=1}^g T_{c_i}$ in $\Mod(S_g)$. 

It is well known~\cite{H1,AW} that the largest two possible orders of a torsion element in $\Mod(S_g)$ are $4g+2$ and $4g$. In fact, when $n$ is $4g$ or $4g+2$, there exists an order-$n$ mapping class $W_n$ that is realized~\cite{PKS} as the $2\pi/n$ rotation of a regular hyperbolic $n$-gon with opposite sides identified. As a final application of our theory we show the following. 
\begin{prop*}
\label{prop1}
For $g \geq 2$, there exists conjugates $$W_{4g}'=(T_{b_1}^2\prod_{i=1}^{g-1}(T_{c_i}T_{b_{i+1}})T_{a_g}) \text{ and } W_{4g+2}'=(T_{a_1}\prod_{i=1}^{g-1}(T_{b_i}T_{c_i})T_{b_g})$$ of $W_{4g}$ and $W_{4g+2}$, respectively, such that $W_{4g}'W_{4g+2}'$ is pseudo-Anosov in $\Mod(S_g)$.
\end{prop*}

\noindent It may be noted that the first part of Proposition~\ref{prop1} was also shown by Ishizaka in~\cite{MI}.

\section{Preliminaries}\label{sec:prelim}
The primary purpose of this section is to summarize the theory developed in~\cite{BPR,PKS}, which lies at the foundation of the various methods that we will develop in this paper.
\subsection{Periodic mapping classes}\label{sec:finite_maps}
For $g \geq 1$, let $F \in \Mod(S_g)$ be of order $n$. By the Nielsen-Kerckhoff theorem~\cite{SK,JN},  $F$ is represented by a \textit{standard representative} $\F \in \Homeo^+(S_g)$ of order $n$. Let $\O_F := S_g/\langle \F \rangle$ be the \textit{corresponding orbifold} of $F$ of genus $g_0$ (say). Each cone point $x_i \in \O_F$ lifts under the branched cover $S_g \to S_g/\langle \F \rangle$ to an orbit of size $n/n_i$ on $S_g$, where the local rotation induced by $\F$ is given by $2 \pi c_i^{-1}/n_i$,  $c_i c_i^{-1} \equiv 1 \pmod{n_i}$. The tuple $\Gamma(\O_F) := (g_0; n_1,\ldots,n_{\ell})$, is called the \textit{signature of} $\O_F$. By the theory of group actions on surfaces (see~\cite{H1} and the references therein), we obtain an exact sequence: 
\begin{equation*}
\label{eq:surf_kern}
1 \rightarrow \pi_1(S_g) \rightarrow \pi_1^{orb}(\O_F) \xrightarrow{\Phi_F} \langle \F \rangle \rightarrow 1, 
\end{equation*}
\noindent where the orbifold fundamental group $\pi_1^{orb}(\O_F)$ of $\O_F$ has the presentation
\begin{equation*}
\label{eqn:orb-pres}
\left\langle \alpha_1,\beta_1,\dots,\alpha_{g_0},\beta_{g_0}, \xi_1,\dots,\xi_{\ell} \, |\, \xi_1^{n_1},\dots,\xi_\ell^{n_{\ell}},\,\prod_{j=1}^{\ell} \xi_j \prod_{i=1}^{g_0}[\alpha_i,\beta_i]\right\rangle
\end{equation*} 
and $\Phi_{F} (\xi _i) = \F^{(n/n_i)c_i}$, for $1 \leq i \leq \ell$. We will now define a tuple of integers that will encode the conjugacy class of a periodic mapping class $F \in \Mod(S_g)$ of order $n$ in $\Mod(S_g)$.  

\begin{defn}\label{defn:data_set}
A \textit{data set of degree $n$} is a tuple
$$
D = (n,g_0, r; (c_1,n_1),\ldots, (c_{\ell},n_{\ell})),
$$
where $n\geq 2$, $g_0 \geq 0$, and $0 \leq r \leq n-1$ are integers, and each $c_i \in \Z_{n_i}^\times$ such that:
\begin{enumerate}[(i)]
\item $r > 0$ if and only if $\ell = 0$ and $\gcd(r,n) = 1$, whenever $r >0$,
\item each $n_i\mid n$,
\item $\lcm(n_1,\ldots \widehat{n_i}, \ldots,n_{\ell}) = N$, for $1 \leq i \leq \ell$, where $N = n$, if $g_0 = 0$,  and
\item $\displaystyle \sum_{j=1}^{\ell} \frac{n}{n_j}c_j \equiv 0\pmod{n}$.
\end{enumerate}
The number $g$ determined by the Riemann-Hurwitz equation
\[\frac{2-2g}{n} = 2-2g_0 + \sum_{j=1}^{\ell} \left(\frac{1}{n_j} - 1 \right) \tag{R-H} \]
is called the {genus} of the data set, denoted by $g(D)$.
\end{defn}

\noindent The following proposition, which allows us to use data sets to represent the conjugacy classes of cyclic actions on $S_g$, is well known in literature (see~\cite{AI,SH}, and~\cite[Theorem 3.89]{KP}). However, it follows mainly from a result of Nielsen~\cite{JN1}.

\begin{prop}\label{prop:ds-action}
For $g \geq 1$ and $n \geq 2$, data sets of degree $n$ and genus $g$ correspond to conjugacy classes of $\Z_n$-actions on $S_g$. 
\end{prop}

\noindent We will denote the data set encoding the conjugacy class of a periodic mapping class $F$ by $D_F$. The parameter $r$ (in Definition~\ref{defn:data_set}) will come into play in a data set $D_F$ only when $\F$ is a free rotation of $S_g$ by $2\pi r/n$, in which case, $D_F$ will take the form $(n,g_0,r;)$. We will avoid including $r$ in the notation of a data set $D_F$, whenever $\F$ is non-free.  Furthermore, for compactness of notation, we also write a data set $D$ (as in Definition~\ref{defn:data_set}) as
$$D = (n,g_0,r; ((d_1,m_1),\alpha_1),\ldots,((d_{\ell'},m_{\ell'}),\alpha_{\ell'})),$$
where $(d_i,m_i)$ are the distinct pairs in the multiset $S = \{(c_1,n_1),\ldots,(c_{\ell},n_{\ell})\}$, and the $\alpha_i$ denote the multiplicity of the pair $(d_i,m_i)$ in the multiset $S = \{(c_1,n_1),\ldots,(c_{\ell},n_{\ell})\}$. For simplicity of notation, we shall avoid writing the parameters $\alpha_i$ when they equal $1$. 

Let $F \in \Mod(S_g)$ be of order $n$. Then $F$ is said to be \textit{rotational} if $\F$ is a rotation of the $S_g$ through an axis by $2 \pi r/n$, where $\gcd(r,n)=1$. It is apparent that $\F$ is either has no fixed points, or $2k$ fixed points which are induced at the points of intersection of the axis of rotation with $S_g$. Moreover, these fixed points will form $k$ pairs of points $(x_i,x_i')$, for $1 \leq i \leq k$, such that the sum of the angles of rotation induced by $\F$ around $x_i$ and  $x_i'$ add up to $0$ modulo $2\pi$. Consequently, we have the following: 
\begin{prop}
\label{prop:rotl_actn}
Let $F \in \Mod(S_g)$ be a rotational mapping class of order $n$. 
\begin{enumerate}[(i)]
\item When $\F$ is a non-free rotation, then $D_F$ has the form
 $$(n,g_0;\underbrace{(s,n),(n-s,n),\ldots,(s,n),(n-s,n)}_{k \,pairs}),$$ for integers $k \geq 1$ and $0<s\leq n-1$ with $\gcd(s,n)= 1$, and $k=1$ if $n>2$.  
 \item When $\F$ is a free rotation, then $D_F$ has the form
 $$(n,\frac{g-1}{n}+1,r;).$$
\end{enumerate}
\end{prop}

\noindent We say $F$ is of \textit{Type 1} if $\Gamma(\O_F)$ has the form $(g_0; n_1,n_2,n)$, and $F$ is said to be of \textit{Type 2} if $F$ is neither rotational nor of Type 1. Gilman~\cite{G1} showed that a periodic mapping class $F \in \Mod(S_g)$ is irreducible if and only if $\O_F$ is a sphere with three cone points. Thus, $F$ is an irreducible Type 1 mapping class if and only if $\Gamma(\O_F)$ has the form $(0; n_1,n_2,n)$.

\subsection{Decomposing periodic maps into irreducibles} 
\label{sec:pri_into_irred}
In ~\cite{BPR,PKS}, a method was described to construct an arbitrary non-rotational periodic element $F \in \Mod(S_g)$, for $g \geq 2$, by performing certain``compatibilties" on irreducible Type 1 actions. These Type 1 actions are in turn realized as rotations of certain unique hyperbolic polygons with side-pairings. 
\begin{theorem}[{\cite[Theorem 2.7]{PKS}}]
\label{res:1}
For $g \geq 2$, consider an irreducible Type 1 action $F \in {\Mod}(S_g)$ with $$D_F = (n,0; (c_1,n_1),\linebreak (c_2,n_2), (c_3,n)).$$ Then $F$ can be realized explicitly as the rotation $\theta_F= 2\pi c_3^{-1}/n$ of a hyperbolic polygon $\P_F$ with a suitable side-pairing $W(\P_F)$, where $\P_F$ is a hyperbolic  $k(F)$-gon with
$$ k(F) := \begin{cases}
2n, & \text { if } n_1,n_2 \neq 2, \text{ and } \\
n, & \text{otherwise, }
\end{cases}$$
and for $0 \leq m\leq n-1$, 
$$ 
W(\P_F) =
\begin{cases}
\displaystyle  
  \prod_{i=1}^{n} a_{2i-1} a_{2i} \text{ with } a_{2m+1}^{-1}\sim a_{2z}, & \text{if } k(F) = 2n, \text{ and } \\
\displaystyle
 \prod_{i=1}^{n} a_{i} \text{ with } a_{m+1}^{-1}\sim a_{z}, & \text{otherwise,}
\end{cases}$$
where $\displaystyle z \equiv m+qj \pmod{n}$ with $q= (n/n_2)c_3^{-1}$ and $j=n_{2}-c_{2}$.
\end{theorem} 

\noindent Consequently, the process yielded a concrete construction of a hyperbolic metric invariant under the cyclic action. Further, it was shown that the process of decomposition can be reversed by piecing together the irreducible Type 1 components (described in Theorem~\ref{res:1}) using the methods that we will now describe in Constructions~\ref{cons:k-comp} and \ref{cons:perm_add}. 

\begin{cons}[$k$-compatibility]
\label{cons:k-comp}
For $i=1,2$, let $F_i \in \Mod(S_{g_i})$ be of order $n$. Suppose that the actions of $\langle \F_i \rangle$ on the $S_{g_i}$  induce a pair of orbits $O_i$ such that $|O_1| = |O_2|$ and the rotation angles induced by the $\langle \F_i \rangle$-action around points in the $O_i$ add up to $0 \pmod{2\pi}$. Then we remove (cyclically permuted) $\langle \F_i \rangle$-invariant disks around points in the $O_i$ and then attach $k$-annuli $A_i$ connecting the resulting boundary components, to obtain an $F \in \Mod(S_{g})$ of order $n$, where $g(F):=g = g_1+g_2+k-1$. This method of constructing $F$ is called a \textit{k-compatibility}, and we say that \textit{$F$ is realizable as a $k$-compatible pair $(F_1,F_2)$ of genus $g(F)$.} Further, we denote $A(F) := \sqcup_{i=1}^k A_i$. A typical $1$-compatibility between irreducible Type 1 maps is illustrated in Figure~\ref{fig:Fix_glue} below. (For a visualization of a $k$-compatibility for $k\geq2$, see Figure~\ref{fig:chain_connect}.) 
 
\begin{figure}[H]
		\centering
	\includegraphics[width=20ex]{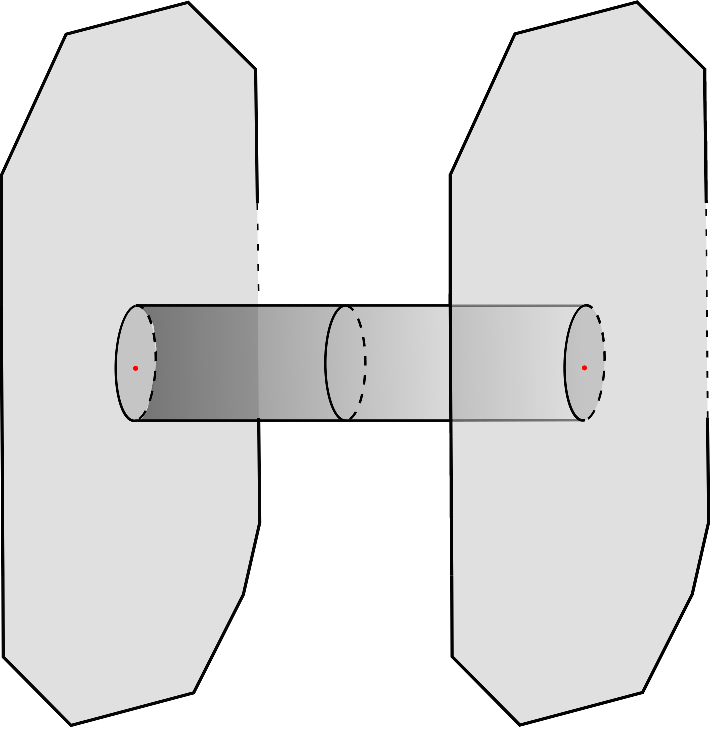}
	\caption{A $1$-compatibility of a pair of irreducible Type 1 maps.}
	\label{fig:Fix_glue}

\end{figure}

\noindent Let $\Sigma_i(F) := \overline{S_{g_i} \setminus A(F)}$. Then by construction, the maps $\F \vert_{\Sigma_i(F)}$ and $\F \vert_{A(F)}$ commute with each other.

If in the construction above, the orbits $O_i$ are induced by a single action on a surface $S_g$, then the method is called a \textit{self k-compatibility}, wherein the resultant action is on $S_{g+k}$. 
\end{cons}

\noindent Generalizing the ideas in Construction~\ref{cons:k-comp}, we have the following. 
\begin{defn}
\label{defn:compk_tuple}
Let $F \in \Mod(S_g)$ be of order $n$. We say $F$ is a \textit{linear $s$-tuple $(F_1,F_2,\ldots, F_s)$ of degree $n$ and genus $g$} if for $1 \leq i \leq s$, there exists $F_i \in \Mod(S_{g_i})$ of order $n$ satisfying the following conditions.
\begin{enumerate}[(i)]
\item $F_{i,i+1}:=(F_i,F_{i+1})$ is realizable as a $k_i$-compatible pair of genus $g(F_{i,i+1})$, for $1 \leq i \leq s-1$.
\item Let 
$$\Sigma_i(F):= 
\begin{cases}
\overline{S_{g_i} \setminus A(F_{i,i+1})}, & \text{if }i=1, \\
\overline{S_{g_i} \setminus A(F_{i-1,i})}, & \text{if }i=s, \text{ and } \\
\overline{S_{g_i} \setminus (A(F_{i-1,i}) \sqcup A(F_{i,i+1}))}, & \text{for } 2 \leq i \leq s-1.
\end{cases}$$
\noindent Then $\F \vert_{\Sigma_i(F)} = \F_i \vert_{\Sigma_i(F)}$, for $ 1 \leq i \leq s$.
\item $S_g = \sqcup_{i=1}^{s} \Sigma_i(F)\sqcup_{i=1}^{s-1} A(F_{i,i+1}),$ where  $$g = \sum_{i=1}^s g_i + \sum_{i=1}^{s-1}(k_i-1).$$
\end{enumerate} 
\end{defn}

\noindent Given a linear $s$-tuple $F=(F_1,F_2,\ldots, F_s)$ as in Definition~\ref{defn:compk_tuple}, we denote $g(F) :=g$, and further, we fix the following notation that for $1 \leq i < j -1 < s$, we denote  
$$F_{i,j} := (F_i,F_{i+1}, \ldots, F_j) \text{ and } \Sigma_{i,j}(F) := \displaystyle \cup_{k=i}^j \Sigma_k(F).$$

\begin{cons}[Permutation additions and deletions]
\label{cons:perm_add}
 The \textit{addition of a $g'$-permutation component} to a periodic map $\F$ is a process that involves the removal of (cyclically permuted) invariant disks around points in an orbit of size $n$ and then pasting $n$ copies of $S_{g'}^1$ (i.e. $S_{g'}$ with one boundary component) to the resultant boundary components. This realizes an action on $S_{g+ng'}$ with the same fixed point and orbit data as $\F$. A visualization of a permutation addition to an irreducible Type 1 map is shown in Figure~\ref{fig:perm_add} below.
 \begin{figure}[H]
	\includegraphics[width=30ex]{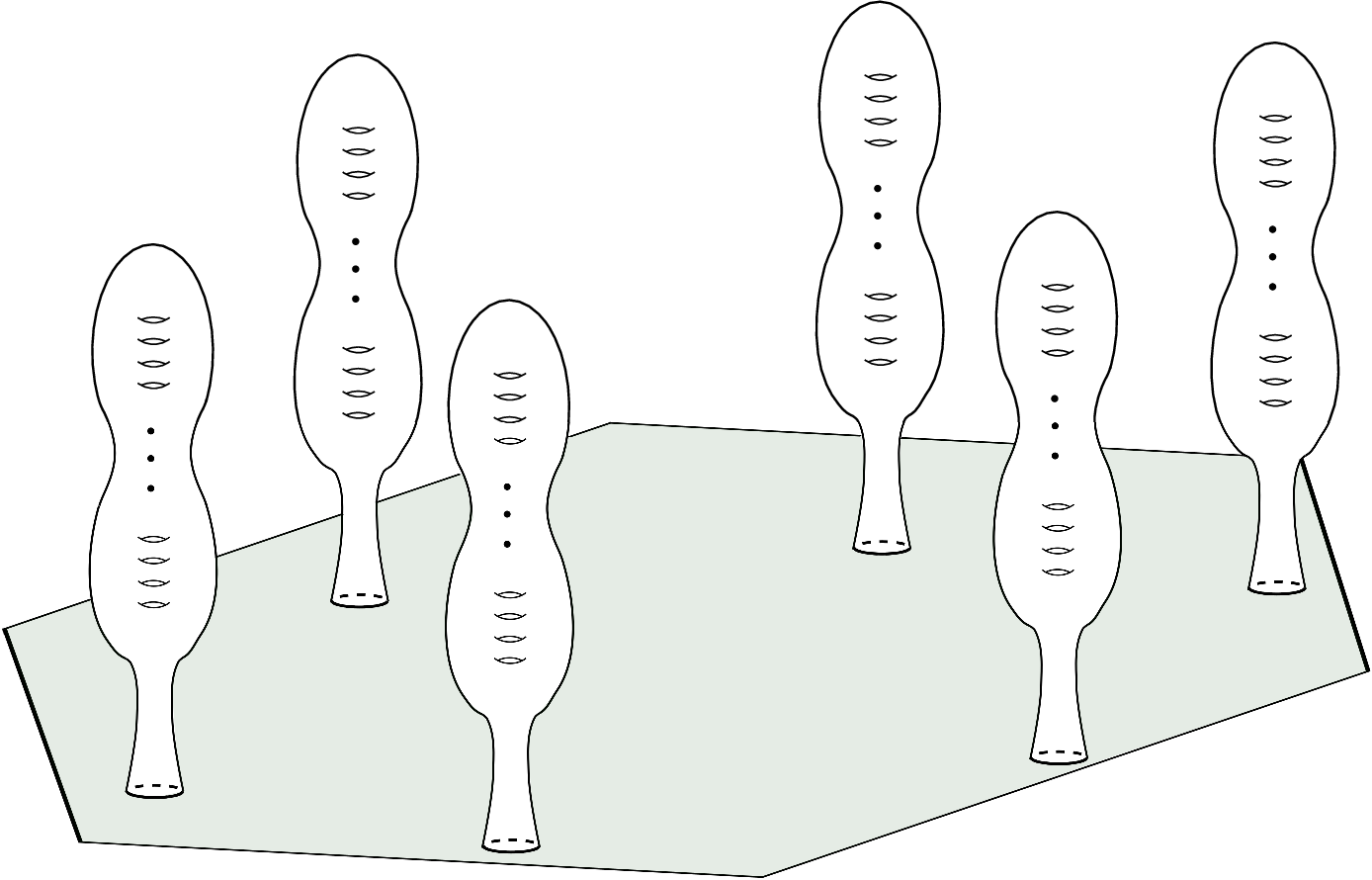}
	\caption{Addition of a $g'$-permutation component to an irreducible Type 1 map.}
	\label{fig:perm_add}
\end{figure}
\noindent The reversal of this process, wherein a $g'$-permutation is removed from $\F$ (when possible), is called the \textit{deletion of $g'$-permutation component.}
\end{cons}
 
\noindent The upshot of the discussion above and~\cite[Lemma 2.23]{PKS} is the following:
\begin{theorem}[{\cite[Theorem 2.24]{PKS}}]
\label{res:2}
For $g \geq 2$, an arbitrary non-rotational periodic mapping class in $\Mod(S_g)$ can be constructed through finitely many $k$-compatibilities, permutation additions, and permutation deletions on irreducible Type 1 mapping classes. 
\end{theorem}

\begin{rem}
It is important to note that Theorem~\ref{res:2} does not imply that a periodic mapping class always decomposes into Type 1 irreducibles. In fact, any irreducible Type 2 mapping class cannot be decomposed any further. However, the theorem does guarantee (see ~\cite[Lemma 2.23]{PKS}) that an arbitrary Type 2 mapping class can be built inductively from finitely many irreducible Type 1 mapping classes through the processes described in Constructions~\ref{cons:k-comp} and~\ref{cons:perm_add}. For example, consider the irreducible Type 2 $Z_{30}$-action $$D = (30, 0; (1, 6), (1, 10), (11, 15))$$ on $S_{11}$. Consider the following $\mathbb{Z}_{30}$-actions: 
\begin{enumerate}[(i)]
\item $D_1 = (30, 0; (11, 15), (19, 30), (19, 30))$, 
\item $D_2 = (30, 0; (1, 6), (7, 15), (11, 30))$, and
\item $D_3 = (30, 0; (1, 10), (8, 15), (11, 30))$. 
\end{enumerate}
First we note that the $\mathbb{Z}_{30}$-action $D' = (30, 1; (1, 6), (1, 10), (11, 15))$ on $S_{41}$ can be built from $1$-compatibility on the pair $(D_1,D_2)$, followed by a $2$-compatibility on the pair $(D_2,D_3)$, and then a $1$-compatibility on the pair $(D_3,D_1)$. Finally, we perform a $1$-permutation deletion on $D'$ to realize $D$. In general, \cite[Lemma 2.23]{PKS} says that any irreducible Type 2 mapping class can be constructed in a similar manner. 
\end{rem}

\begin{defn}
\label{defn:admi_tuple}
Given integers $s>0$, $u,v,w \geq 0$, an \textit{admissible $(s,u,v,w)$-tuple $\T$} is a tuple of integers of the form
\begin{gather*}
\T = [((i_1,j_1),k_1),\ldots,((i_u,j_u),k_u);(i_1',g_1'),\ldots, (i_v',g_v'); \\ ((i_1'',j_1''),g_1''), \ldots,((i_w'',j_w''),g_w'')], 
\end{gather*}
where for each $q$, $1 \leq i_q < j_q \leq s$, $k_q \geq 1$, , $1 \leq i_q' \leq s$, $1 \leq i_q'' < j_q'' \leq s$, and $g_q',g_q'' > 0$.
\end{defn}

\begin{defn}
\label{defn:comp_FTtuple}
Given a linear $s$-tuple $(F_1,\ldots, F_s)$ of degree $n$ and genus $g$ as in Definition~\ref{defn:compk_tuple} and an admissible $(s,u,v,w)$-tuple $\T$ as in Definition~\ref{defn:admi_tuple}, we construct a \textit{compatible $(F,\T)$-tuple} $F_{\T}$ of degree $n$ and genus $g(F_{\T})$ through the constructions in the following sequence of steps.
\begin{enumerate}[\textit{Step} 1.]
\item If $u = 0$, then we skip this step. Otherwise, for $1 \leq q \leq u$, we perform a self $k_q$-compatibility in $\F_{i_q,j_q}$, if $\F_{i_q,j_q}$ admits such a compatibility.
\item If $v = 0$, then we skip this step. Otherwise, for $1 \leq q \leq v$, we perform a $g_q'$-permutation addition on the $\F_{i_q}$.
\item If $w = 0$, then we skip this step. Otherwise, for $1 \leq q \leq w$, we perform a $g_q''$-permutation deletion on the $\F_{i_q,j_q}$, if $\F_{i_q,j_q}$ admits such a deletion.
\end{enumerate}
\end{defn}

\noindent Note that a compatible $(F,\T)$-tuple, where $\T$ is an admissible $(s,0,0,0)$-tuple simply refers to the linear $s$-tuple $F$. With this notation in place, Theorem~\ref{res:2} can be now restated as follows. 
\begin{theorem}\label{res:21}
Given an arbitrary non-rotational periodic mapping class in $G \in \Mod(S_g)$, for $g \geq 2$, there exists a linear $s$-tuple $F \in \Mod(S_{g'})$ of irreducible Type 1 actions and an admissible $(s,u,v,w)$-tuple $\T$ of integers such that $G = F_{\T}$.
\end{theorem}

\subsection{Symplectic representations of periodic mapping classes} 
\label{sec:per_symp_rep}
For $g \geq 1$, let $\Psi : \Mod(S_g) \to \text{Sp}(2g;\Z)$ be the surjective representation afforded by the action of $\Mod(S_g)$ on $H_1(S_g;\Z)$ with respect to the \textit{canonical homology generators} $\{a_i,b_i : 1 \leq i \leq g\}$ as shown in Figure~\ref{fig:2_lick_gens}. In this subsection, we will state some results from~\cite[Section 4]{PKS} that are relevant to this paper.

Let $F \in \Mod(S_g)$ be an irreducible Type 1 action that is realized by the rotation of a hyperbolic polygon $\P_F$ with a boundary word $W(\P_F)$ (when read counterclockwise) as in Theorem~\ref{res:1}. Let $W(\P_F)$ be of the form $Q a R b S a^{-1} T b^{-1} U,$ for some words $Q, R, S, T,U$ (possibly empty), and letters $a,b$. Then by an application of the \emph{handle normalization algorithm} detailed in~\cite[Section 3.4]{GR}, we have the following.

\begin{prop}[{\cite[Proposition 4.2]{PKS}}]\label{prop:normal}
Let $W(\P_F) =QaRbSa^{-1}Tb^{-1}U$. Suppose that $\P'$ is the polygon with boundary word $W(\P') = QTSRUxyx^{-1}y^{-1}$ obtained by applying the handle normalization algorithm once to $\P_F$. Then $x$ and $y$ are freely homotopic to $U^{-1}R^{-1}a^{-1}Tb^{-1}U$ and $QTb^{-1}U$, respectively. 
\end{prop}

\noindent We note that there is a small typographical error in the way Proposition~\ref{prop:normal} (i.e. Proposition 4.2) is stated in~\cite{PKS} as the roles of $x$ and $y$ are mistakenly interchanged in the statement. However, this error does not appear in the proof.

By a set of \textit{standard generators} of $H_1(S_g;\Z)$, we mean collection $\{a_i',b_i': 1 \leq i \leq g\}$ of curves occurring along the $4g$-sided polygon $\P(g)$ with $\W(\P(g)) = \prod_{i=1}^g [a_i',b_i']$. In order to compute $\Psi(F)$ (up to conjugacy), we need to express a set of \textit{standard generators} $\{a_i',b_i': 1 \leq i \leq g\}$ of $H_1(S_g;\Z)$ in terms of the letters occurring along the boundary of the polygon $\P_F$ (as in Theorem~\ref{res:1}) that realizes an irreducible Type 1 action $F$. To this end, we apply the handle normalization algorithm iteratively on the polygon $\P_F$. Before moving ahead, we first fix the following notation. 
\begin{enumerate}[(a)]
\item We denote by $\N^i(\P_F)$, the polygon obtained from $\P_F$ after $i$ successive applications of the normalization procedure described in Proposition~\ref{prop:normal}.
\item We denote by $L(\P_F)$, the set of distinct letters in $W(\P_F)$. 
\item We denote by $\B(\P_F)$, a set of standard generators $\{a_i',b_i': 1 \leq i \leq g\}$ of $H_1(S_g;\Z)$ obtained from iterative applications of the handle normalization algorithm, expressed in terms of the elements in $L(\P_F)$. 
\end{enumerate}

\noindent Let $W=W(\P_F)$ and $W' = W(\P')$ be as in Proposition~\ref{prop:normal}. Then the relabelling map $$ \B(\P')\to \B(\P_F)\, :\, x \mapsto U^{-1}R^{-1}a^{-1}Tb^{-1}U, \, y \mapsto QTb^{-1}U ,\, z \mapsto z,$$ $z\in \B(\P')\setminus \{x, y\}$ yields a change of basis matrix $f_{\P',\P_F}$ which is obtained by expressing each of vectors in $\B(\P')$ in terms of the letters in $W(P_F)$. This leads us to the following lemma.

\begin{lemma}[{\cite[Lemma 4.4]{PKS}}]
\label{lem:normalization}
Let $F \in \Mod(S_g)$ be an irreducible Type 1 action that is realized by the rotation of a hyperbolic polygon $\P_F$ as in Theorem~\ref{res:1}.  Then 
$$W(\N^g(\P_F)) = \displaystyle \prod_{i=1}^g [x_i,y_i],$$ and the matrix
 $$\displaystyle f_{\P_F}= \prod_{i=1}^{g} f_{\N^i(\P_F),\,\N^{i-1}(\P_F)}$$ defines a change of basis matrix on $H_1(S_g;\mathbb{Z})$ such that $$\B(\P(g)) \xmapsto{f_{\P_F}} \B(\P_F).$$ 
\end{lemma}

\noindent  For an isomorphism $\varphi: H_1(S_g;\mathbb{Z}) \to H_1(S_g;\mathbb{Z})$, let $M_{\varphi}$ denote the matrix of $\varphi$ with respect to a set of standard homology generators. The following theorem describes the structure of $\Psi(F)$, up to conjugacy.

\begin{theorem}[{\cite[Theorem 4.5]{PKS}}]
\label{thm:rep_irrtype1}
Let $F \in \Mod(S_g)$ be an irreducible Type 1 with $D_F = ((n,0;(c_1,n_1), (c_2,n_2), (c_3,n)).$ Then  up to conjugacy, $\displaystyle \Psi(F) = M_{\varphi},  \text{ where } \varphi = f_{\P_F}^{-1}\phi_{\P_F}f_{\P_F},$ with $f_{\P_F}$ as in Lemma ~\ref{lem:normalization}, and $\B(\P_F) \xmapsto{\phi_{\P_F}} \B(\P_F)$ is induced by $a_i \mapsto a_j,$ where
$$j \equiv \begin{cases}
i+2c_3^{-1} \pmod{2n} , & \text{if } n_1,n_2\neq 2, \text{ and}\\
i+c_3^{-1} \pmod{n},  & \text{otherwise}.
\end{cases}
$$
\end{theorem}

\noindent We refer the reader to Example~\ref{eg:symp_rep_order9_map} for an illustration of Theorem~\ref{thm:rep_irrtype1}. We conclude this subsection with the following remark. 

\begin{rem}
\label{rem:irr_symp_blocks}
By Theorem~\ref{res:2}, an arbitrary non-rotational periodic mapping class $F \in \Mod(S_g)$ can be decomposed into irreducible Type 1 mapping classes. This decomposition induces a decomposition of $\Psi(F)$ (up to conjugacy) into a block-diagonal matrix, where each diagonal block is of one of the following types.
\begin{enumerate}[(i)]
\item The image under $\Psi$ of an irreducible Type 1 component (of $F$) as described in Theorem~\ref{thm:rep_irrtype1}. 
\item Let $F'$ be a component of $F$ resulting from a $k$-compatibility (or a self $k$-compatibility), and let $S_g^b$ denote the surface of genus $g$ with $b$ boundary components. Then there exists a subsurface $S$ (of $S_g$) homeomorphic to $S_{k-1}^2$ (shown in Figure~\ref{fig:chain_connect} below) in which $\F'$ cyclically permutes the disjoint union of the $k$ annuli $A_F \subset S$ involved in the construction. 
\begin{figure}[H]
\centering
\labellist
		\small
		\pinlabel $S_{k-1}^2$ at 398 20
		\endlabellist
	\includegraphics[width=45ex]{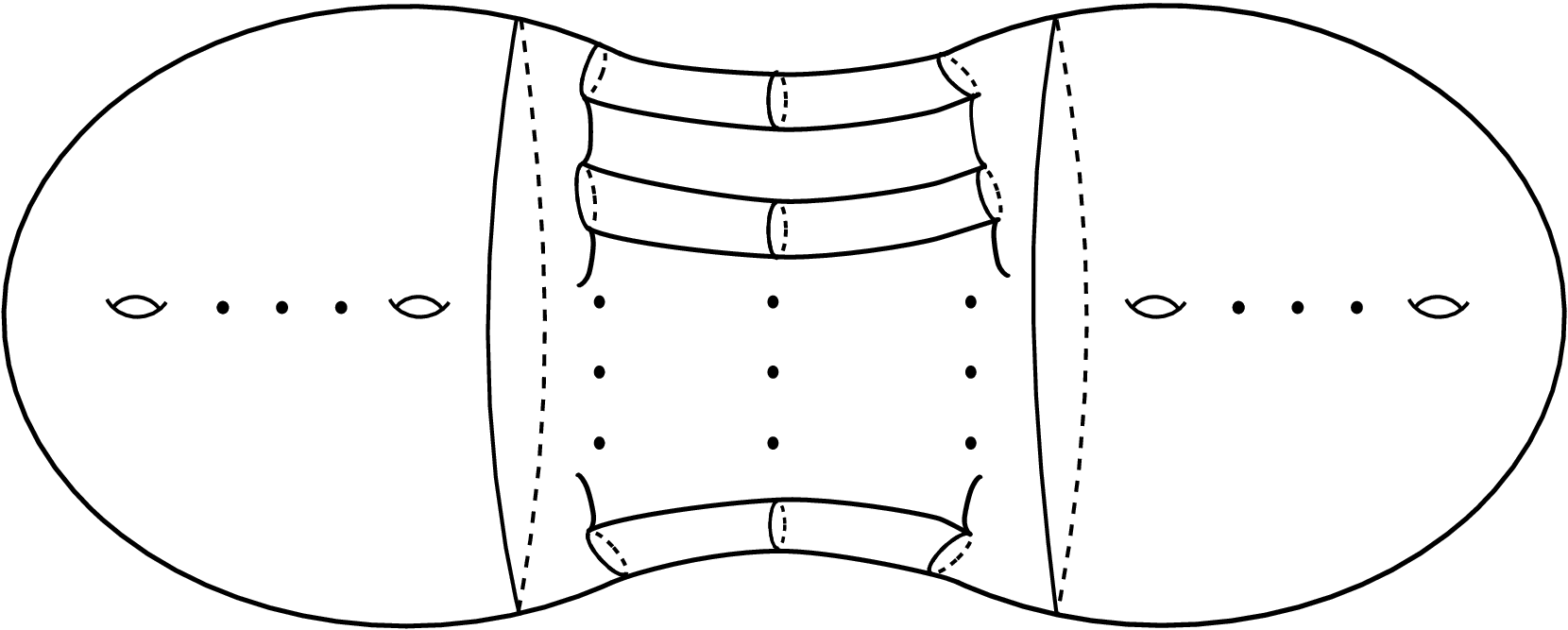}
	\caption[A $k$-compatibility of a pair of irreducible Type 1 maps.]{The subsurface $S \approx S_{k-1}^2$.}
	\label{fig:chain_connect}
\end{figure}
\noindent The diagonal block is obtained from the well-defined action of such an $F'$ on $H_1(S,\Z)$.
\item The image under $\Psi$ under a permutation component of $F$, which permutes $n$ subsurfaces of $S_g$ homeomorphic to some $S_{g'}^1$ as in Construction~\ref{cons:perm_add}. 
\end{enumerate}
\noindent Note that as the blocks of type (ii) and (iii) are simple permutation blocks, one can obtain a  complete description of $\Psi(F)$ (up to conjugacy). 
\end{rem}

\subsection{Relations involving Dehn twists}\label{sec:reln_dehn_twist}
Let $i(c,d)$ denote the geometric intersection number of simple closed curves $c$ and $d$ in $S_g$. A collection $\C=\{c_1,\ldots,c_k\}$ of simple closed curves in $S_g$ is said to form a \textit{chain} if $i(c_i,c_{i+1}) =1$, for $1 \leq i \leq k-1$, and $i(c_i,c_j) = 0$, if $|i-j|>1$. We state the following basic fact (see ~\cite[Section 1.3]{FM}) about chains. 
\begin{lemma}
\label{lem:exist_chains}
For $g \geq 1$, there is a unique chain in $S_g$ of size $2g$, up to homeomorphism. Moreover, when $g > 1$, there is a unique chain in $S_g$ of size $2g+1$, up to homeomorphism. 
\end{lemma}

A closed regular neighborhood of the union of curves in $\C$ is a subsurface $S$ of $S_g$ that has one or two boundary components, depending on whether $k$ is even or odd. Let the isotopy classes of $\partial S$ be represented by the curves $d$ (resp. $d_1,d_2$) when $k$ is even (resp. odd). Let $T_c$ denote the left-handed Dehn twist about a simple closed curve $c$ in $S_g$. We will make extensive use of the well known chain relation involving Dehn twists.
\begin{prop}[Chain relation]
\label{prop:chain_reln}
Let $\C = \{c_1,\ldots,c_k \}$ be a chain in $S_g$, where $2 \leq k \leq 2g+1$. Then:
$$ \begin{array}{lcll}
(T_{c_1}T_{c_2} \ldots T_{c_k})^{2k+2} & = & T_d, & \text{ when } k \text{ is even, and} \\ 
(T_{c_1}T_{c_2} \ldots T_{c_k})^{k+1} & = & T_{d_1}T_{d_2}, & \text{ when } k \text{ is odd.} 
\end{array}$$ 
Equivalently, we have 
$$ \begin{array}{lcll}
(T_{c_1}^2 T_{c_2} \ldots T_{c_k})^{2k} & = & T_d, & \text{ when } k \text{ is even, and} \\ 
(T_{c_1}^2T_{c_2} \ldots T_{c_k})^{k} & = & T_{d_1}T_{d_2}, & \text{ when } k \text{ is odd.}
\end{array}$$
\end{prop}

\noindent In each case of Proposition~\ref{prop:chain_reln}, we will denote the word enclosed within the parentheses on the left hand side by $W_{\C}$. Note that for all $1<i<k$, $W_{\C}(c_i) = c_{i+1}$. We will also use the following relations also known as the \textit{hyperelliptic relations} (see \cite[Chapter 5]{FM}) in $\Mod(S_g)$, for $g \geq 2$. 

\begin{prop}
\label{prop:hyp_relns}
For $g \geq 2$, let $\{c_1,\ldots,c_{2g+1}\}$ be a chain in $S_g$. Then:
$$\begin{array}{rcl}
(T_{c_{2g+1}} \ldots T_{c_1} T_{c_1} \ldots T_{c_{2g+1}})^2 & = & 1,\text{ and} \\ 
\left[ T_{c_{2g+1}} \ldots T_{c_1} T_{c_1} \ldots T_{c_{2g+1}}, T_{c_{2g+1}} \right]& = & 1,\\
\end{array}$$ 
where $T_{c_{2g+1}} \ldots T_{c_1} T_{c_1} \ldots T_{c_{2g+1}}$ represents the conjugacy class of the standard hyperelliptic involution in $\Mod(S_g)$ encoded by $(2,0; ((1,2),2g+2))$.
\end{prop}
 
\noindent  Let $c_1,c_2,c_3,d_1,d_2,d_3$, and $b$ be the curves in $S_1^3$, as indicated in Figure~\ref{fig:star_reln} below.
\begin{figure}[H]
		\labellist
		\tiny
		\pinlabel $\textcolor{black}{b}$ at 133 180
		\pinlabel $\textcolor{black}{c_1}$ at 100 87
		\pinlabel $\textcolor{black}{c_2}$ at 270 60
		\pinlabel $\textcolor{black}{c_3}$ at 165 210
		\pinlabel $\textcolor{black}{d_1}$ at 50 150
		\pinlabel $\textcolor{black}{d_2}$ at 211 78
		\pinlabel $\textcolor{black}{d_3}$ at 358 150
		\endlabellist
		\includegraphics[width=35ex]{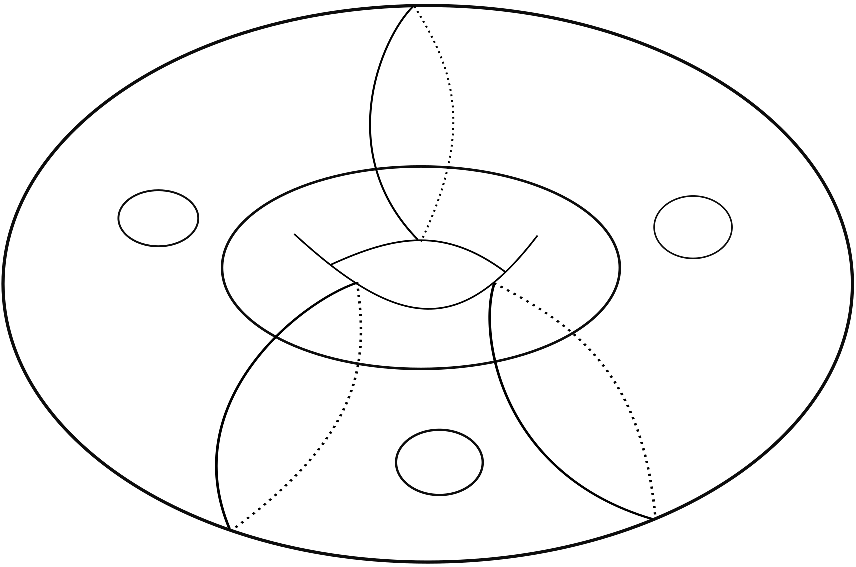}
		\caption{The curves involved in the star relation in $\Mod(S_1^3)$.}
		\label{fig:star_reln}
\end{figure}
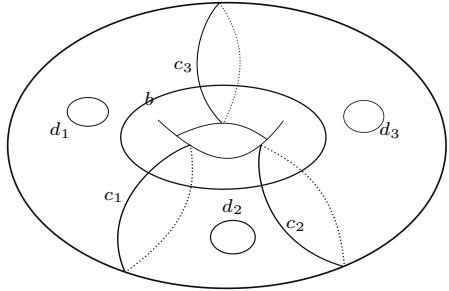
\noindent We will use the following relation  in $\Mod(S_1^3)$ due to Gervais~\cite{SG}, also known as the \textit{star relation}, to develop a method for writing periodic mapping classes of order $3$ as words in Dehn twists. 

\begin{prop}[Star relation]
\label{prop:star_reln}
Let $c_1,c_2,c_3,d_1,d_2,d_3$, and $b$ be the curves in $S_1^3$, as indicated in Figure~\ref{fig:star_reln}. Then:
$$(T_{c_1}T_{c_2}T_{c_3}T_b)^3 = T_{d_1}T_{d_2}T_{d_3}.$$
\end{prop}
\noindent In Section~\ref{sec:gen_star_method}, we will derive a generalization of this relation, which we will apply to develop a method for obtaining the word representations for a larger family of periodic maps. The final result we state in this subsection pertains to the Burkhardt \textit{handle swap} map~\cite{HB,MS1}, which swaps the $i^{th}$ handle in $S_g$ (for $g \geq 2$) with the $(i+1)^{th}$ handle.

\begin{prop}
\label{prop:hand_swap}
For $g \geq 2$, let $a_1,b_1,\ldots, a_g,b_g$ be the curves that represent the canonical generators of $H_1(S_g;\Z)$. Then for $1 \leq i \leq g-1$, the `` $i^{th}$ handle" swap map is given by 
$$H_{i+1,i} : = (T_{a_{i+1}}T_{b_{i+1}}T_{x_i}T_{a_i}T_{b_i})^3,$$
where $x_i$ is a simple closed curve that represents the homology class $a_{i+1}+b_i$.
\end{prop}

\subsection{Periodic maps on the torus as words in Dehn twists}\label{sec:per_torus} Since $\Mod(S_1) \cong \text{SL}(2,\Z) = \Z_4 \ast_{\Z_2} \Z_6$, any non-trivial periodic element in $\Mod(S_1)$ is of order $2$, $3$, $4$, or $6$. Moreover, since $\{a,b\}$ (as indicated in Figure~\ref{hom_torus} below) is a chain in $S_1$, it follows by Proposition~\ref{prop:chain_reln} that $T_aT_b$ is of order $6$, and $T_a^2 T_b$ is of order $4$ in $\Mod(S_1)$. Note that $T_aT_b$ (resp. $T_a^2T_b$) is represented by a rotation of a regular hexagon (resp. square) with opposite sides identified, by $2 \pi/6$ (resp. $\pi/2)$. 
\begin{figure}[H]
\includegraphics[width=25ex]{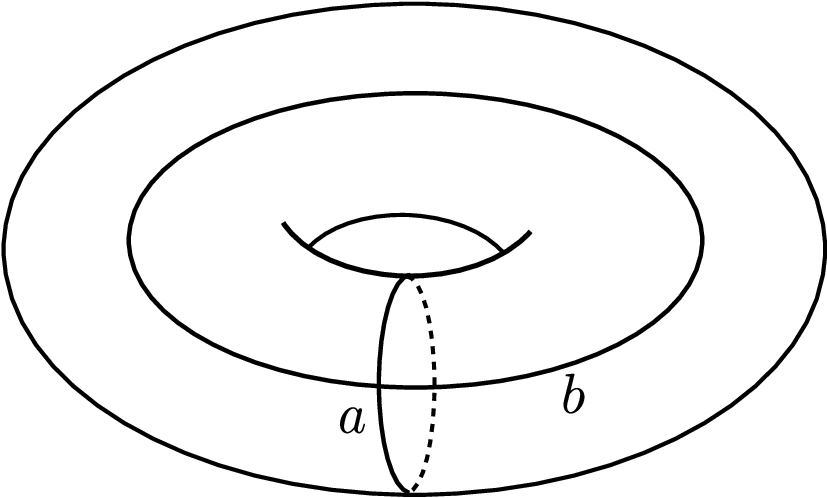}
\caption{A chain in the torus.}
\label{hom_torus}
\end{figure}
\noindent Taking the powers of these maps, we obtain a word $\W(F)$ (in Dehn twists) representing the conjugacy class of each periodic element $F \in \Mod(S_1)$.

\begin{table}[H]
\begin{center}
\begin{tabular}{|c|c|c|}
\hline	$|F|$ & $D_F$ & $\W(F)$\\
\hline	6 & $(6,0;(1,2),(1,3),(1,6))$ & $T_aT_b$  \\
\hline    6 & $(6,0;(1,2),(2,3),(5,6))$ & $(T_aT_b)^5$  \\
\hline    4 & $(4,0;(1,2),(1,4),(1,4))$  &  $T_a^2T_b$\\
\hline    4 & $(4,0;(1,2),(3,4),(3,4))$  &  $(T_a^2T_b)^3$\\
\hline	3 & $(3,0;(1,3),(1,3),(1,3))$  & $(T_aT_b)^2$\\
\hline    3 & $(3,0;(2,3),(2,3),(2,3))$  & $(T_aT_b)^4$\\
\hline	2 & $(2,0;((1,2),4))$ & $(T_aT_b)^3$\\
\hline
\end{tabular}
\end{center}
\caption{Words (in Dehn twists) representing the conjugacy classes of periodic elements in $\Mod(S_1)$.}
\label{tab:torus_words}
\end{table}

\section{Rotational mapping classes as words in Dehn twists}\label{sec:rotation_word}
In this section, we will provide a method for writing rotational mapping classes as products of Dehn twists. The key idea is to write given rotational mapping class as a product of two involutions, whose representations (as words) will be discussed in the following subsection.

\subsection{Non-free involutions as words in Dehn twists}
By Proposition~\ref{prop:rotl_actn}, given an arbitrary involution $F \in \Mod(S_g)$, $D_F$ has one of the following forms:
$$(2,g_0;((1,2),2k)) \text{ or } (2,(g+1)/2,1;),$$ depending on whether $\F$ is non-free or free. First, we consider the cases $g=1,2$, where there are three possible conjugacy classes of involutions. 
\begin{enumerate}[(a)]
	\item The hyperelliptic involution in $\Mod(S_1)$: $(2,0;((1,2),4))$.
	\item The hyperelliptic involution in $\Mod(S_2)$:  $(2,0;((1,2),6))$.
	\item The rotation of $S_2$ with two fixed points: $(2,1;((1,2),2))$.
\end{enumerate}
The word representation for the involution in $(a)$ was featured in Table~\ref{tab:torus_words}, while the word for (b), the hyperelliptic involution, is known from Proposition~\ref{prop:hyp_relns}. Since $(c)$ swaps the two genera of $S_2$, it is the map $H_{2,1}$ from Proposition~\ref{prop:hand_swap}. We will collectively call involutions in (a) and (c) (above) as the \textit{fundamental involutions}. We will show that an arbitrary involution can be obtained by piecing together the fundamental involutions via $1$-compatibilities. We will require the following lemmas, which are simple consequences of Proposition \ref{prop:chain_reln}. 

\begin{lemma}
	\label{lem:S22}
	Let \(\overline{H_{2,1}}\) be the restriction of \(H_{2,1}\) on \(S_2^2\). Then, \[\overline{H_{2,1}}^2 = T_{d_1} T_{d_2},\] where \(d_1, d_2\) are the boundary curves of \(S_2^2\) as shown in Figure \ref{fig:inv_S22} below.
	\begin{figure}[H]
	\labellist
	\tiny
		\pinlabel $\textcolor{black}{d_1}$ at 400 190
		\pinlabel $\textcolor{black}{d_2}$ at 410 118
		\pinlabel $\textcolor{black}{a_1}$ at 240 110
		\pinlabel $\textcolor{black}{b_1}$ at 110 148
		\pinlabel $\textcolor{black}{x_1}$ at 140 75
		\pinlabel $\textcolor{black}{a_2}$ at 520 110
		\pinlabel $\textcolor{black}{b_2}$ at 660 148
		\endlabellist
	\includegraphics[width=35ex]{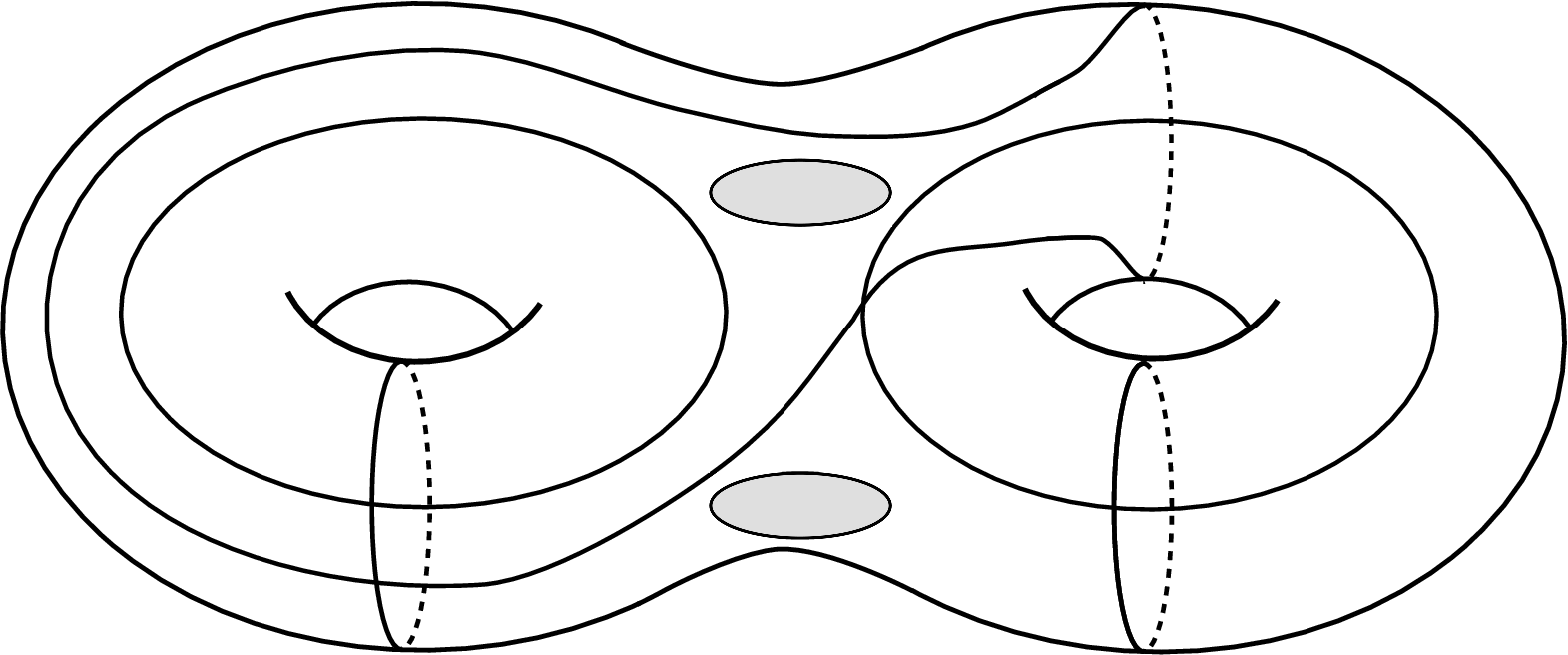}
	\caption{The boundary curves $d_1,d_2$ in \(S_2^2.\)}
	\label{fig:inv_S22}
\end{figure}
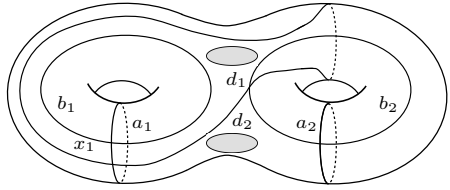
\end{lemma}

\begin{proof}
First, we note that the boundary of a regular neighborhood of the chain  $(a_2,b_2,x_1,a_1,b_1)$ is homotopic to $d_1 \sqcup d_2$, as indicated in Figure~\ref{fig:inv_S22}. Thus, by Proposition~\ref{prop:chain_reln}, we have $(T_{a_2}T_{b_2}T_{x_1}T_{a_1}T_{b_1})^6=T_{d_1}T_{d_2}$, from which the assertion follows. 
\end{proof}

\noindent From a similar argument as above, we have the following lemma. 

\begin{lemma}
	\label{lem:S12}
	Let $a,b,a',d_1,d_2$ be the curves in $S_1^2$ as indicated in the Figure \ref{fig:inv_S12} below. 
	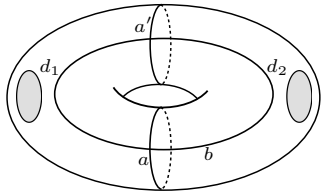
\begin{figure}[H]
	\labellist
		\tiny
		\pinlabel $\textcolor{black}{a'}$ at 178 220
		\pinlabel $\textcolor{black}{a}$ at 178 45
		\pinlabel $\textcolor{black}{b}$ at 260 55
		\pinlabel $\textcolor{black}{d_2}$ at 347 165
		\pinlabel $\textcolor{black}{d_1}$ at 60 165
		\endlabellist
	\includegraphics[width=25ex]{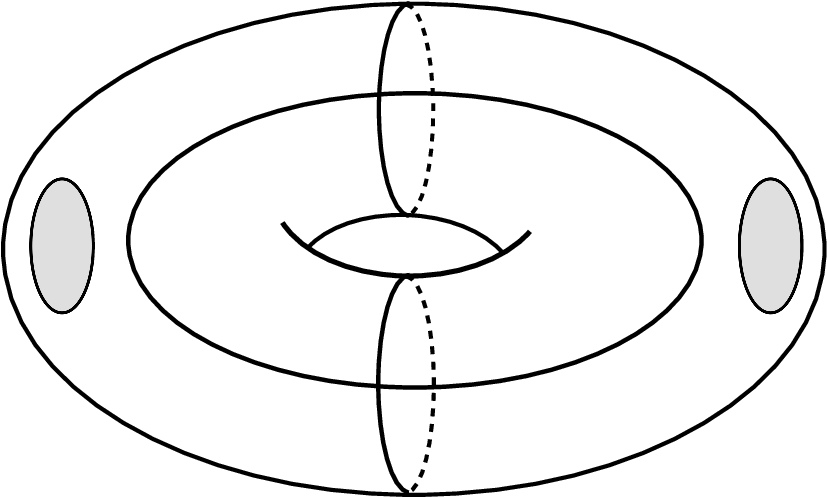}
	\caption{The curves $a,b,a',d_1,d_2$ in $S_1^2$.}
	\label{fig:inv_S12}
\end{figure}
Then 
 $$(T_a T_b T_{a'})^4 = T_{d_1}T_{d_2}.$$
\end{lemma}

\noindent We will now provide an algorithm for writing involutions as words in the Dehn twists. 

\begin{algo}
\label{alg:inv}
Let $F \in \Mod(S_g)$ be a non-free involution with $D_F = ((2,g_0;((1,2),2k))$ (by virtue of Lemma~\ref{prop:rotl_actn}).
\begin{enumerate}[\textit{Step} 1.]
	\item If $k=1$, then: 
	\begin{enumerate}
	\item[\textit{Step} 1a.] We decompose $\F$ into fundamental involutions as shown in the Figure~\ref{fig:inv_tower2}.
	\begin{figure}[H]
	\includegraphics[width=35ex]{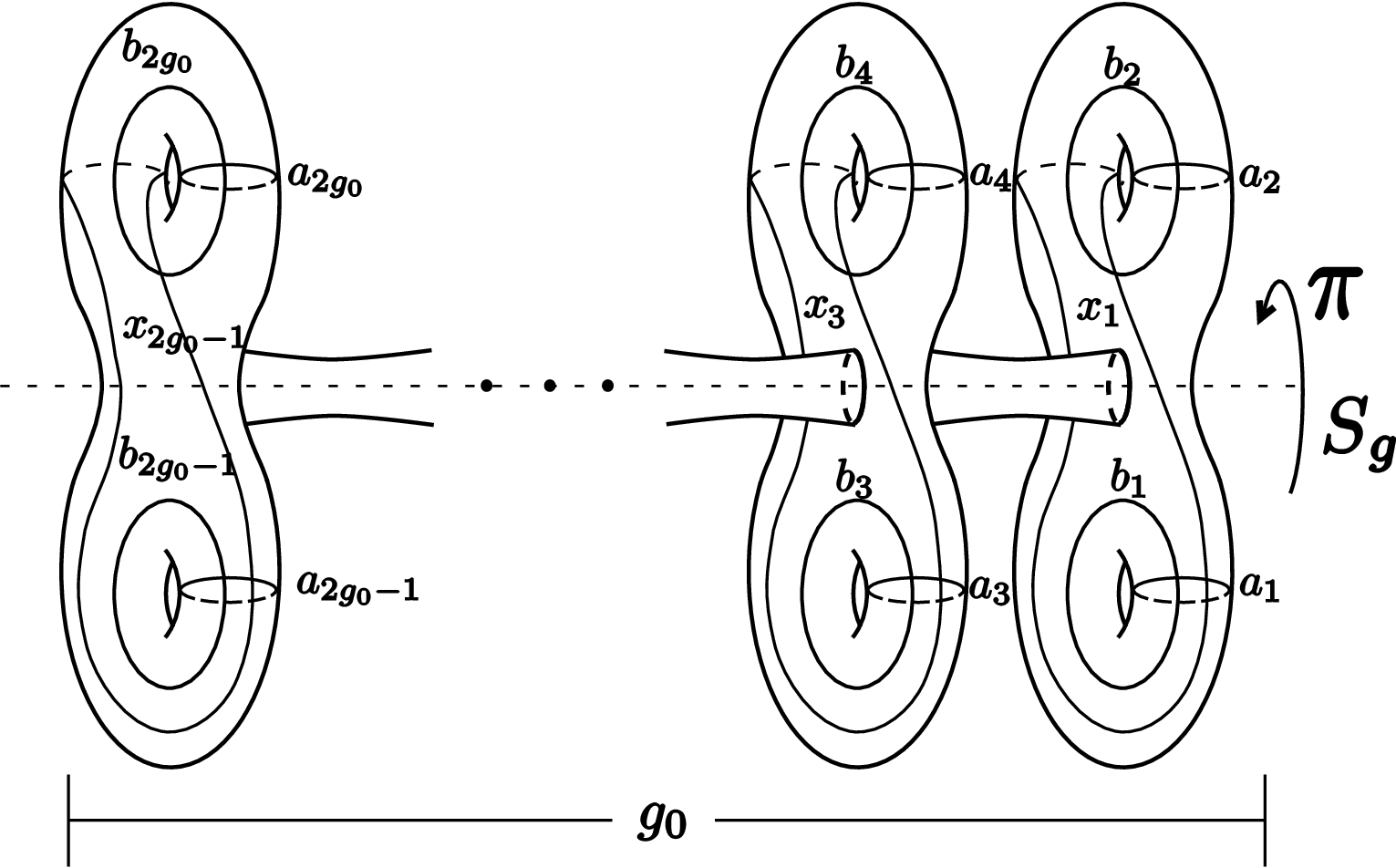}
	\caption{Decomposition of $\F$ for $k=1$ into fundamental involutions.}
	\label{fig:inv_tower2}
	\end{figure}
\item[\textit{Step} 1b.] We set 
$$\W(F) = \displaystyle \prod_{i = 1}^{g_0} (H_{2i,2i-1})^{(-1)^{i-1}}.$$
\end{enumerate}
\item If $k > 1$, then:
\begin{enumerate}[\textit{Step} 2a.]
	\item[\textit{Step} 2a.] We decompose $\F$ into fundamental involutions as shown in the Figure~\ref{fig:inv_tower}. 
	\begin{figure}[H]
	\centering
	\includegraphics[width=55ex]{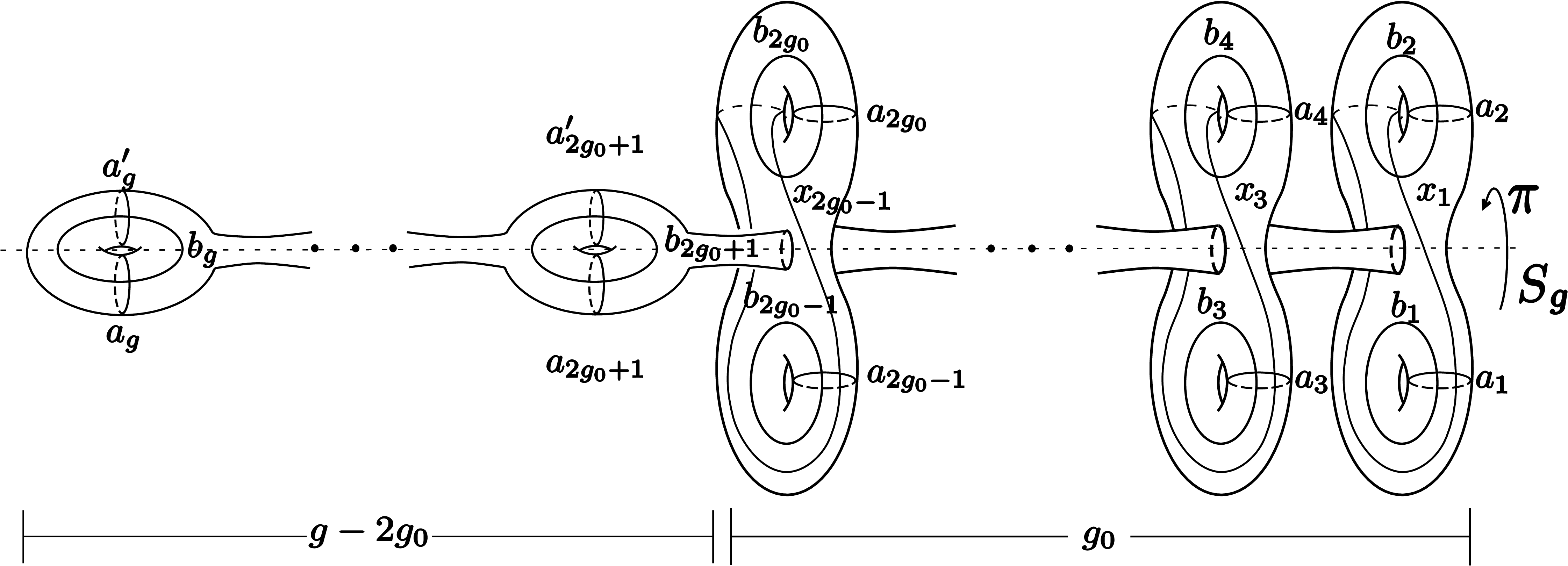}
	\caption{Decomposition of $\F$ for $k >1$ into fundamental involutions.}
	\label{fig:inv_tower}
\end{figure}
\item[\textit{Step} 2b.] We set
$$\W(F)=\displaystyle \prod_{i = 1}^{g_0} (H_{2i,2i-1})^{(-1)^{i-1}}\prod_{j = 2g_0+1}^{g} (T_{a_j}T_{b_j}T_{a_j'})^{2(-1)^{j+g_0-1}},$$ 
\end{enumerate}
\item By Proposition~\ref{prop:hand_swap} and Lemmas~\ref{lem:S22}-\ref{lem:S12}, $\W(F)$ is the desired representation of $F$ as a word in Dehn twists, up to conjugacy.
\end{enumerate}
\end{algo}

\noindent A simple computation reveals that in general, Algorithm~\ref{alg:inv} will express an arbitrary non-free involution (up to conjugacy) as a word in $3g-2+ \lfloor g/2 \rfloor$ Dehn twists about nonseparating curves. However, it is important to note each application of Algorithm~\ref{algo:surf_rotn} may involve up to $3g-2+{g \choose 2}$ (distinct) Dehn twists about nonseparating curves. 

\subsection{Surface rotations as words in Dehn twists}  For $g \geq 2$, any rotation of $S_g$ (that is free or non-free) of order $n \geq 2$ can be written as a product of two involutions, as illustrated in Figure~\ref{fig:two_inv_rot}. 
\begin{figure}[htbp]
\includegraphics[width=25ex]{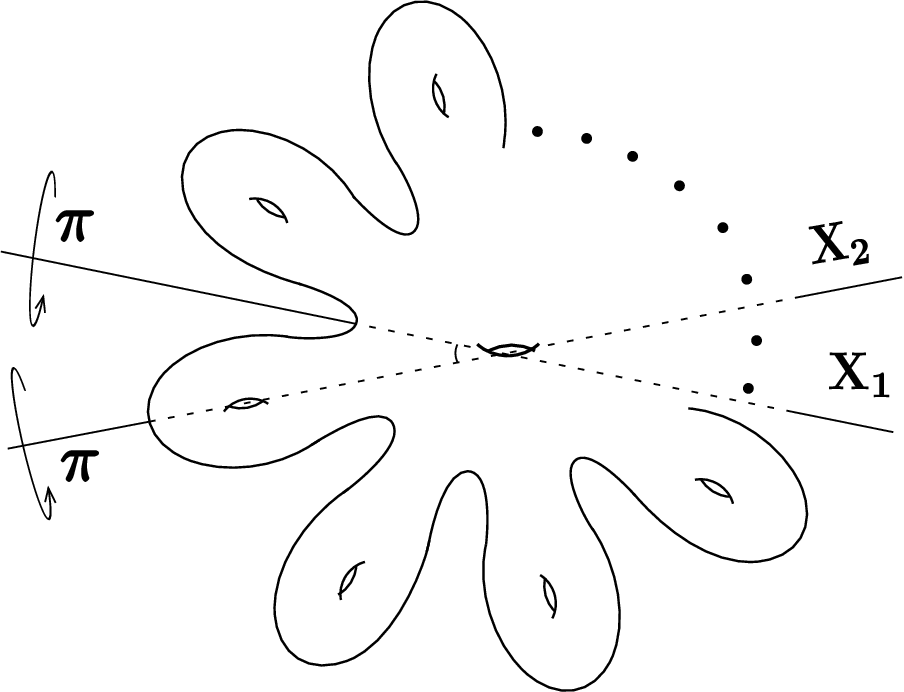}
\caption{A surface rotation as a product of two involutions.}
\label{fig:two_inv_rot}
\end{figure}
\noindent This leads to the following method for writing surface rotations as words in Dehn twists. 

\begin{algo}
\label{algo:surf_rotn}
Let $F \in \Mod(S_g)$ be realized as a rotation $\F$ of order $n > 2$ or as a free involution. Then by Proposition~\ref{prop:rotl_actn}, $D_F$ has the form $$(n,g_0;\underbrace{(s,n),(n-s,n),\ldots,(s,n),(n-s,n)}_{k \,pairs}) \text{ or } (n,\frac{g-1}{n}+1,r;),$$ depending on whether $\F$ is free rotation or not.
\begin{enumerate}[\textit{Step} 1.]
	\item  Consider an embedding of $S_g$ in $\mathbb{R}^3$, as indicated in Figure~\ref{fig:two_inv_rot}, where there is a ``genus in the middle", only when $\F$ is free.
	\item For $i =1,2$, let the reflection along the axis $X_i$ (as shown in the figure) be $\Theta_i$, where $\Theta_i$ is a non-free involution determined by Algorithm~\ref{alg:inv}. We set $R_g = \Theta_1 \cdot \Theta_2$. 
	\item If $\F$ is free, then we set $\W(F) = R_g^{r(g-1)/n}$, else we set $\W(F) = R_g^{\frac{gs^{-1}}{n}}$.
	\item $\W(F)$ is the desired representation of $F$ as a word in Dehn twists, up to conjugacy.
\end{enumerate}
\end{algo}

\section{Chain method}
\label{sec:chain_method}
 In this section, we provide a method by which one can write certain periodic mapping classes as words in Dehn twists by repeated application of the chain relation. Let $F \in \Mod(S_g)$ be an irreducible Type 1 mapping class. Let $F \in \Mod(S_g)$ be of order $n$, and let $(1,n)$ be a pair in $D_F$ representing a fixed point of the $\langle \F \rangle$-action on $S_g$. Now consider the mapping class $F^{m}$, for some integer $1 \leq  m \leq |F|$. Then in $D_{F^m}$, there exists a pair $(c',n')$ (representing a fixed point of the $\langle \F^m \rangle$-action on $S_g$) that originated from the pair $(1,n)$ such that $n' = |F^m| = n/\gcd(m,n)$ and $(c')^{-1} \equiv m/\gcd(m,n) \pmod{n'}$. We will denote this pair $(c',n')$ in $D_{F^m}$ by $(1,n)_{m,F}$.

\begin{defn}
\label{defn:chain_real}
Let $F \in \Mod(S_g)$ be realizable as a linear $s$-tuple $(F_1,\ldots,F_s)$ of degree $n$ and genus $g$ as in Definition~\ref{defn:compk_tuple}. Then $F$ is said to be \textit{chain-realizable} if 
$F$ admits a realization as a linear $s$-tuple $(F_1,\ldots,F_s)$ of genus $g$ such that the following conditions hold. 
\begin{enumerate}[(i)]
\item For $1 \leq i \leq s$, there exists an irreducible Type 1 mapping class $\tilde{F}_i \in \Mod(S_{g_i})$, a filling chain $\C(\tilde{F}_i)$ in $S_{g_i}$, and an $m_i \geq 1$ such that $F_i$ is conjugate to $(W_{\C(\tilde{F}_i)})^{m_i}$. Then:
\begin{enumerate}
\item For each $i$, $D_{\tilde{F}_i}$ has one of the following forms on $S_{g_i}$
\begin{enumerate}[1.]
\item $(2{g_i}+1,0;(2{g_i}-1,2{g_i}+1),(1,2{g_i}+1),(1,2{g_i}+1))$
\item $(2{g_i}+2,0;({g_i},{g_i}+1),(1,2{g_i}+2),(1,2{g_i}+2))$,
\item $(4{g_i},0;(1,2),(1,4{g_i}),(2{g_i}-1,4{g_i}))$, 
\item $(4{g_i}+2,0;(1,2),({g_i},2{g_i}+1),(1,4{g_i}+2))$
\end{enumerate}
\item  For $1 \leq i \leq s-1$,  $k_i =1$, and for each pair $(F_i,F_{i+1})$, the $1$-compatibility is across a pair of fixed points represented by pairs of the form $(c_i,n)$ (in $D_{F_i}$) and $(n-c_i,n)$ (in $D_{F_{i+1}}$), where $(c_i,n) = (1, |\tilde{F}_i|)_{m_i,\tilde{F}_i} $ and $(n-c_i,n) =  (1, |\tilde{F}_{i+1}|)_{{m_{i+1},\tilde{F}_{i+1}}}$.
\end{enumerate}
\end{enumerate}
\end{defn}

\noindent It is worth mentioning here that the factorizations of the mapping classes listed in (i)(a) Definition~\ref{defn:chain_real} were first derived in~\cite{MI}.
\begin{defn}
\label{defn:chain_realgen}
A periodic mapping class $G \in \Mod(S_g)$ is said to be \textit{chain-realizable} if there exists a chain-realizable linear $s$-tuple $F \in \Mod(S_g)$ and a nonzero integer $q$ such that $G = F^q$.
\end{defn}

\noindent Given $c \in \mathbb{Z}_n^{\times}$, we will fix the following notation. 
\begin{enumerate}[(a)]
\item $c^+ = c(+1) := \{d \in \Z : cd \equiv 1 \pmod{n}\} \cap [0,n].$
\item $c^- = c(-1) := \{d \in \Z : cd \equiv 1 \pmod{n}\} \cap [-n,0].$
\end{enumerate}

\begin{lemma}
\label{lem:chain_real}
Let $F \in \Mod(S_g)$ be realizable as a chain-realizable $s$-tuple of degree $n$ and genus $g$ as in Definition~\ref{defn:chain_real}. For all $i$, let $\W(F_i)=(W_{\C(\tilde{F}_i )})^{\beta_i \bar{c_i}}$, where $\bar{c}_i = c_i((-1)^{i+1})$ and $\beta_i=\frac{|\tilde{F}_i|}{|F_i|}$. Then:
$$\W(F) = \prod_{i=1}^s \W(F_i)$$ is conjugate to $F$.
\end{lemma}

\begin{proof}
By Proposition~\ref{prop:chain_reln}, for each $i$, $((W_{\C(\tilde{F}_i )})^{\beta_i \bar{c_i}})^{|F_i|}$ 
equals either $(T_{d_1}T_{d_2})^{\bar{c_i}}$ or $(T_{d_1})^{\bar{c_i}}$, depending upon whether $|\C(\tilde{F}_i )|$ is odd or even. Thus, $\bar{c_i}$ measures the amount of twisting along the boundary of a closed neighborhood of the chain $\C(\tilde{F}_i )$. 
Thus, by Construction~\ref{cons:k-comp} and Definition~\ref{defn:chain_real}, we have that $\W(F)$ is conjugate of $F$. Finally, since $\W(F_i)$ commutes with $\W(F_j)$ for all $1\leq i,j \leq s$, we have $$(\W(F))^n = (\prod_{i=1}^s \W(F_i))^n=\prod_{i=1}^s \W(F_i)^n=\prod_{i=1}^s(T_{d_{i1}}T_{d_{i2}})^{\bar{c}_i}=1,$$ where $d_{i2}$ is taken to be the trivial curve when $i=1,s$.
\end{proof}

\noindent We will now provide an algorithm for representing a chain-realizable periodic mapping classes as words in Dehn twists. 

\begin{algo}[Chain method]
\label{algo:chain_method}
Let $G \in \Mod(S_g)$ be a chain-realizable periodic mapping class.
\begin{enumerate}[\textit{Step} 1.]
\item Write $G = F^q$, where $F$ is a compatible chain-realizable $s$-tuple $(F_1,\ldots,F_s)$ of degree $n$ and genus $g$ as in Definition~\ref{defn:chain_real}. 
\item Set $\W(F_i)=(W_{\C(\tilde{F}_i )})^{\beta_i \bar{c_i}}$, where $\bar{c}_i = c_i((-1)^{i+1})$, and set
$$\W(F) = \prod_{i=1}^s \W(F_i).$$
\item By Lemma~\ref{lem:chain_real}, $\W(G)=\W(F)^q$ is the desired representation of $G$ as a word in Dehn twists, up to conjugacy.
\end{enumerate}
\end{algo}

\begin{exmp}
For $i =1,2$, consider the order $6$ mapping classes $F_i \in \Mod(S_1)$ with $D_{F_1}= (6,0;(1,2),(1,3),(1,6))$ and $D_{F_2} = (6,0;(1,2),(2,3),(5,6))$. The $\F_i$ admit a $1$-compatibility along a pair of compatible fixed points that correspond to the pairs $(1,6)$ and $(5,6)$ in the $D_{F_i}$ where the induced rotation angles are $2\pi/6$ and $10\pi/6$, respectively. This $1$-compatibility yields 
an $F=(F_1,F_2) \in \Mod(S_2)$ with $D_F = (6,0;(1,2),(1,2),(1,3),(2,3))$. If $\C(F_1) = \{a_1,b_1\}$ and $\C(F_2) = \{a_2,b_2\}$, then by Table~\ref{tab:torus_words} and Algorithm~\ref{algo:chain_method}, $F$ is represented up to conjugacy by the word
$$\W(F)=(T_{a_1}T_{b_1})(T_{a_2}T_{b_2})^{-1}.$$
\end{exmp}

\subsection{Periodic maps on $S_2$ as words in Dehn twists} Let $a_1,b_1,c_1,a_2,b_2,$ and $x_1$ be curves in $S_2$, as indicated in Figure~\ref{fig:genus2_curves} below. 

\begin{figure}[H]
\labellist
\small
\pinlabel $\textcolor{black}{c_1}$ at 395 195
\pinlabel $\textcolor{black}{x_1}$ at 405 105
\pinlabel $\textcolor{black}{a_1}$ at 173 115
\pinlabel $\textcolor{black}{a_2}$ at 520 115
\pinlabel $\textcolor{black}{b_1}$ at 273 123
\pinlabel $\textcolor{black}{b_2}$ at 630 129
\endlabellist
\includegraphics[width=40ex]{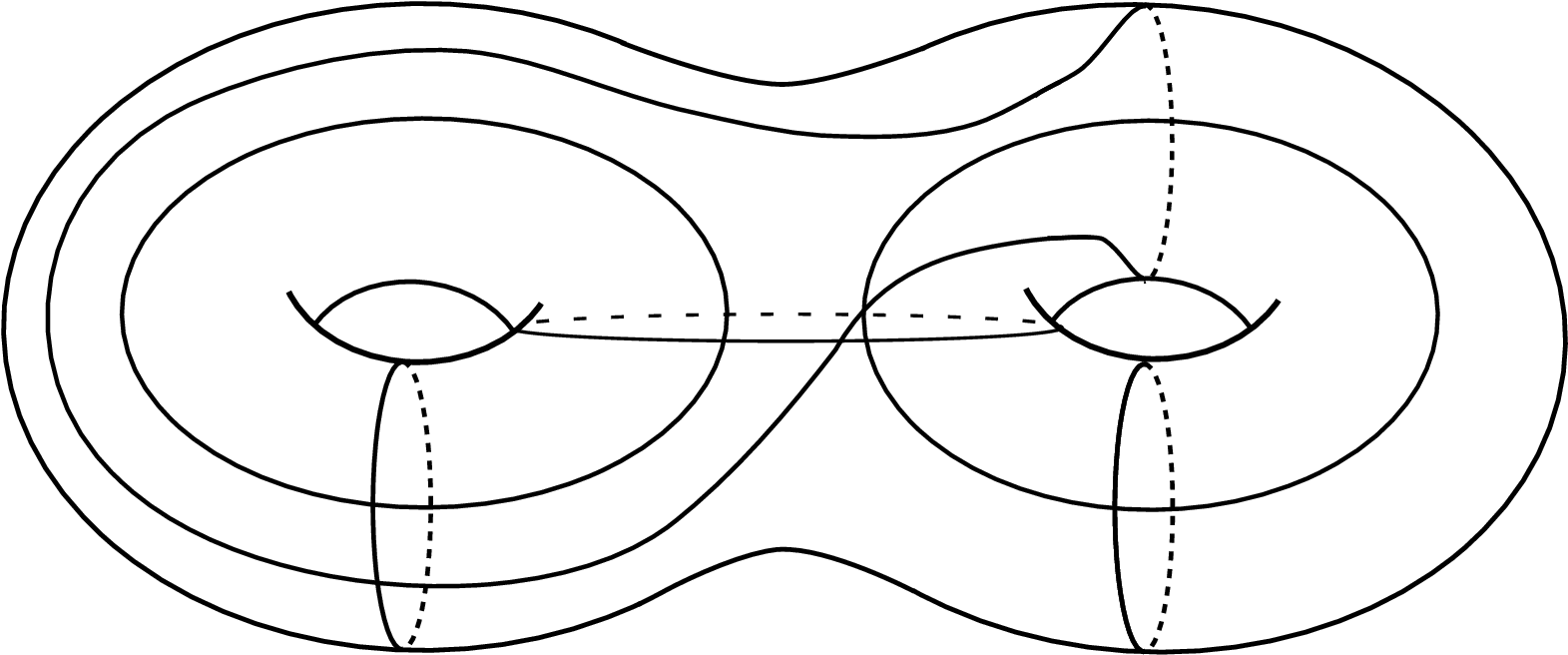}
\caption{Curves $a_1,b_1,c_1,a_2,b_2,$ and $x$ in $S_2$.}
\label{fig:genus2_curves}
\end{figure}
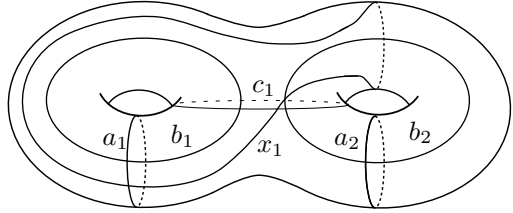

\noindent Using Algorithms~\ref{alg:inv} and~\ref{algo:chain_method}, in Table~\ref{tab:genus2_words} below, we provide a word $\W(F)$ (in Dehn twists) representing the conjugacy class of each periodic element $F \in \Mod(S_2)$. The explicit factorizations of all periodic mapping classes in $\Mod(S_g)$, for $1 \leq g \leq 4$, are also derived in~\cite{SH}.

\begin{table}[htbp]
\begin{center}
\begin{tabular}{|c|c|c|c|}
\hline
$|F|$ & $D_F$ & $\W(F)$ & Algorithm\\
\hline
10 & $(10,0;(1,2),(2,5),(1,10))$&$(T_{a_1}T_{b_1}T_{c_1}T_{b_2})$ & \ref{algo:chain_method} \\
10 & $(10,0;(1,2),(1,5),(3,10))$&$(T_{a_1}T_{b_1}T_{c_1}T_{b_2})^7$ & \ref{algo:chain_method} \\
10 & $(10,0;(1,2),(4,5),(7,10))$&$(T_{a_1}T_{b_1}T_{c_1}T_{b_2})^3$ & \ref{algo:chain_method} \\
10 & $(10,0;(1,2),(3,5),(9,10))$&$(T_{a_1}T_{b_1}T_{c_1}T_{b_2})^9$ & \ref{algo:chain_method} \\
8 & $(8,0;(1,2),(1,8),(3,8))$&$T_{a_1}^2T_{b_1}T_{c_1}T_{b_2}$  & \ref{algo:chain_method} \\
8 & $(8,0;(1,2),(5,8),(7,8))$&($T_{a_1}^2T_{b_1}T_{c_1}T_{b_2})^5$  & \ref{algo:chain_method} \\
6 & $(6,0;((1,2),2),(1,3),(2,3))$& $(T_{a_1}T_{b_1})(T_{a_2}T_{b_2})^{-1}$ & \ref{algo:chain_method} \\
6 & $(6,0;(2,3),(1,6),(1,6))$&$(T_{a_1}T_{b_1}T_{c_1}T_{b_2}T_{a_2})$ & \ref{algo:chain_method} \\
6 & $(6,0;(1,3),(5,6),(5,6))$&$(T_{a_1}T_{b_1}T_{c_1}T_{b_2}T_{a_2})^5$ & \ref{algo:chain_method} \\
5 & $(5,0;((1,5),2),(3,5))$&$(T_{a_1}^2T_{b_1}T_{c_1}T_{b_2}T_{a_2})$ & \ref{algo:chain_method} \\
5 & $(5,0;((2,5),2),(1,5))$&$(T_{a_1}^2T_{b_1}T_{c_1}T_{b_2}T_{a_2})^3$ & \ref{algo:chain_method} \\
5 & $(5,0;((3,5),2),(4,5))$&$(T_{a_1}^2T_{b_1}T_{c_1}T_{b_2}T_{a_2})^2$ & \ref{algo:chain_method} \\
5 & $(5,0;((4,5),2),(2,5))$&$(T_{a_1}^2T_{b_1}T_{c_1}T_{b_2}T_{a_2})^4$ & \ref{algo:chain_method} \\
4 & $(4,0;((1,2),2),(1,4),(3,4))$&$(T_{a_1}T_{b_1}T_{a_1})(T_{a_2}T_{b_2}T_{a_2})^{-1}$ & \ref{algo:chain_method} \\
3 & $(3,0;((1,3),2), ((2,3),2)$&$(T_{a_1}T_{b_1})^2(T_{a_2}T_{b_2})^{-2}$ & \ref{algo:chain_method} \\
2 & $(2,0;((1,2),6))$&$(T_{a_1}T_{b_1}T_{a_1})^2(T_{a_2}T_{b_2}T_{a_2})^{-2}$ & \ref{alg:inv} \\
2 & $(2,1;(1,2),(1,2))$&$(T_{a_2}T_{b_2}T_{x_1}T_{a_1}T_{b_1})^3$ & \ref{alg:inv} \\
\hline
\end{tabular}
\end{center}
\caption{Words (in Dehn twists) representing the conjugacy classes of periodic elements in $\Mod(S_2)$.}
\label{tab:genus2_words}
\end{table}

\subsection{Factoring free involutions} For $g \geq 3$, there exists a unique free involution $F \in \Mod(S_g)$ up to conjugacy, with $D_F = (2,(g+1)/2,1;)$. As another application of the chain method, we will now provide an explicit factorization of such an involution into Dehn twists. 

\begin{prop}
For $g \geq 2$ and $g$ even, let $F \in \Mod(S_{g+1})$ be a free involution. Then $$\W(F)=(T_{a_1}(\prod_{i=1}^{g-1}T_{b_i}T_{c_i})T_{b_g}T_{c_g})^{g+1}T_{a_{g+1}}^{-1}.$$
\end{prop}

\begin{proof}
Consider the periodic mapping class $\bar{F}\in\Mod(S_g)$ of order $2g+2$ with $D_{\bar{F}}=(2g+2,0;(g,g+1),(1,2g+2),(1,2g+2))$, where $g$ is even. Since $\W(\bar{F})=W_{g,2g+2}=T_{a_1}(\prod_{i=1}^{g-1}T_{b_i}T_{c_i})T_{b_g}T_{a_g}$ (see~\cite{MI}), we have $\W(\bar{F}^{g+1}) = W_{g,2g+2}^{g+1}$, where $D_{\bar{F}^{g+1}}=(2,\frac{g}{2};(1,2),(1,2))$. By removing the invariant disks around the fixed points of $\bar{\F}^{g+1}$ and then attaching an annulus connecting the resultant boundary components, we recover $F\in \Mod(S_{g+1})$. Viewing $\W(\bar{F}^{g+1})$ as a mapping class in $\Mod(S_{g+1})$ and renaming curves suitably (ensuring consistency with Figure~\ref{fig:genstarreln}), we apply the chain relation to obtain
$$(T_{a_1}(\prod_{i=1}^{g-1}T_{b_i}T_{c_i})T_{b_g}T_{c_g})^{2g+2}=T_{a_{g+1}}^2,$$
which, in turn, yields the desired expression for $\W(F)$.
\end{proof}

\section{Generalized star method} 
\label{sec:gen_star_method}
In this section, we first derive a generalization of Proposition~\ref{prop:star_reln} for $g \geq 2$. Using this result, we will develop a method to represent a much larger family of periodic mapping classes as words in Dehn twists, as compared with the chain method. As we will see, this family will also encompass the family of periodics described in Definition~\ref{defn:chain_real}. Let $$a_1', a_1, \ldots a_g, b_1,\ldots,b_g,c_1,\ldots,c_{g-1},d_1,d_2, \text{ and }d_3$$ be the isotopy classes of the simple closed curves in $S_g^3$, as shown in Figure~\ref{fig:genstarreln} below.
	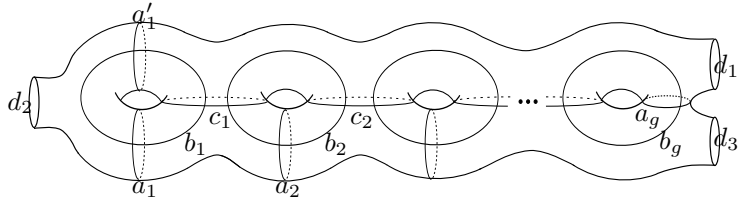
\begin{figure}[H]
	\labellist
		\small
		\pinlabel $\textcolor{black}{a_1}$ at 135 0
		\pinlabel $\textcolor{black}{a_1'}$ at 135 190
		\pinlabel $\textcolor{black}{a_g}$ at 680 77
		\pinlabel $\textcolor{black}{b_g}$ at 705 50
		\pinlabel $\textcolor{black}{b_1}$ at 190 50
		\pinlabel $\textcolor{black}{a_2}$ at 290 0
		\pinlabel $\textcolor{black}{b_2}$ at 342 50
		\pinlabel $\textcolor{black}{c_1}$ at 217 75
		\pinlabel $\textcolor{black}{c_2}$ at 370 75
		\pinlabel $\textcolor{black}{d_2}$ at 0 95
		\pinlabel $\textcolor{black}{d_1}$ at 765 135
		\pinlabel $\textcolor{black}{d_3}$ at 765 55
		\endlabellist
	\includegraphics[width=55ex]{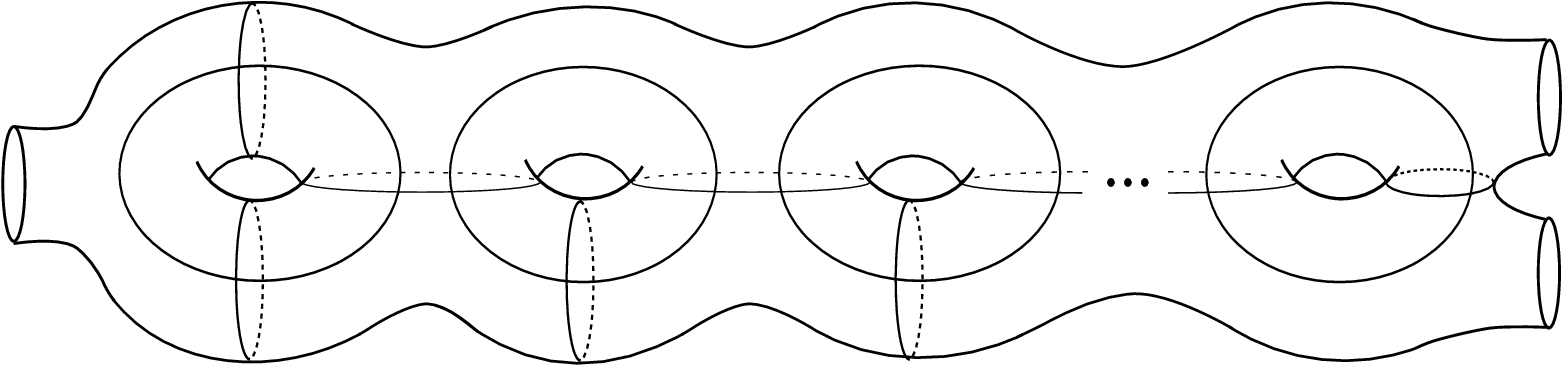}
	\caption{The curves $S_g^3$ involved in the generalized star relation.}
	\label{fig:genstarreln}
\end{figure}

\noindent Note that the curves $a_1$ and $a_1'$ are isotopic in the surface ($\approx S_g^2$) obtained by capping off the boundary curve $d_2$. Further, we consider the surface $S_g^2$ obtained by capping off the boundary curve $d_3$. We have the following generalization of the star relation, which is due to Salter~\cite{NS} (for three boundary components) and~Matsumoto~\cite{M00} (for two boundary components). 

\begin{theorem}[Generalized star relation]
\label{thm:gen_starreln}
For $g \geq 2$ and $k = 2,3$, the following relations hold in $\Mod(S_g^k)$.
\begin{enumerate}[(i)]
\item When $k = 2$, we have: 
$$(T_{a_1}T_{a_1'}\prod_{i=1}^{g-1} (T_{b_i}T_{c_i})T_{b_g})^{4g}=T_{d_2}^{(2g-1)^+}T_{d_1},$$
where $2g-1 \in \Z_{4g}^{\times}$.
\item When $k = 3$, we have: 
$$(T_{a_1}T_{a_1'}\prod_{i=1}^{g-1} (T_{b_i}T_{c_i})T_{b_g}T_{a_g})^{2g+1}=T_{d_2}^{(2g-1)^+}T_{d_1}T_{d_3},$$
where $2g-1 \in \Z_{2g+1}^{\times}$.
\end{enumerate}
\end{theorem}

%

\noindent Clearly, Theorem~\ref{thm:gen_starreln} is a generalization of Proposition~\ref{prop:star_reln}. Moreover, by capping the boundary curve $d_2$, we can also recover  Proposition~\ref{prop:chain_reln}. Following the notation from Section~\ref{sec:chain_method},  we will now introduce a family of periodic mapping classes for which we will develop a method (of deriving $\W(F)$) using Theorem~\ref{thm:gen_starreln}.

\begin{defn}
\label{defn:star_real}
Let $F \in \Mod(S_g)$ be realizable as a linear $s$-tuple $(F_1,\ldots,F_s)$ of degree $n$ and genus $g$ as in Definition~\ref{defn:compk_tuple}. Then $F$ is said to be \textit{star-realizable} if 
$F$ admits a realization as a linear $s$-tuple $(F_1,\ldots,F_s)$ of genus $g$ such that the following conditions hold. 
\begin{enumerate}[(i)]
\item For $1 \leq i \leq s$, there exists an irreducible Type 1 mapping class $\tilde{F}_i \in \Mod(S_{g_i})$, a filling chain $\C(\tilde{F}_i)$ in $S_{g_i}$, and an $m_i \geq 1$ such that $F_i$ is conjugate to $(W_{\C(\tilde{F}_i)})^{m_i}$. Then:
\begin{enumerate}
\item For each $i$, $D_{\tilde{F}_i}$ has one of the following forms on $S_{g_i}$
\begin{enumerate}[1.]
\item $(2{g_i}+2,0;({g_i},{g_i}+1),(1,2{g_i}+2),(1,2{g_i}+2))$,
\item $(4{g_i},0;(1,2),(1,4{g_i}),(2{g_i}-1,4{g_i}))$, 
\item $(4{g_i}+2,0;(1,2),({g_i},2{g_i}+1),(1,4{g_i}+2))$
\item $(2{g_i}+1,0;(2{g_i}-1,2{g_i}+1),(1,2{g_i}+1),(1,2{g_i}+1)) $
\end{enumerate}
\item For $1 \leq i \leq s-1$, $k_i =1$, and for each pair $(F_i,F_{i+1})$, the $1$-compatibility is across a pair of fixed points represented by pairs of the form $(c_i,n)$ (in $D_{F_i}$) and $(n-c_i,n)$ (in $D_{F_{i+1}}$), where $(c_i,n) \in \{(1, |\tilde{F}_i|)_{m_i,\tilde{F}_i}, (|\tilde{F}_i|/2-1, |\tilde{F}_i|)_{m_i,\tilde{F}_i},(|\tilde{F}_i|-2, |\tilde{F}_i|)_{m_i,\tilde{F}_i}\}$ and $(n-c_i,n) \in \{(1, |\tilde{F}_{i+1}|)_{{m_{i+1},\tilde{F}_{i+1}}}, (|\tilde{F}_{i+1}|/2-1, |\tilde{F}_{i+1}|)_{m_{i+1},\tilde{F}_{i+1}},(|\tilde{F}_{i+1}|-2, |\tilde{F}_{i+1}|)_{m_{i+1},\tilde{F}_{i+1}}\}$.
\end{enumerate}
\end{enumerate}
\end{defn}

\noindent We will require the following lemma, which was first proven in~\cite{MI}. (A variant of the same lemma was later shown in~\cite{TN}.) The lemma can also be recovered as a consequence of Theorem~\ref{thm:gen_starreln}.

\begin{lemma}
\label{lem:per_comps}
Consider the periodic mapping classes. 
$$W_{g,j} := \begin{cases} 
T_{a_1}\prod_{i=1}^{g-1} (T_{b_i}T_{c_i})T_{b_g}, & \text{if } j = 4g+2, \\
T_{a_1}T_{a_1'}\prod_{i=1}^{g-1} (T_{b_i}T_{c_i})T_{b_g}, & \text{if } j = 4g,\\
T_{a_1}\prod_{i=1}^{g-1} (T_{b_i}T_{c_i})T_{b_g}T_{a_g}, & \text{if } j = 2g+2,  \text{ and}\\
T_{a_1}T_{a_1'}\prod_{i=1}^{g-1} (T_{b_i}T_{c_i})T_{b_g}T_{a_g}, & \text{if } j = 2g+1.
\end{cases}$$
Then 
$$D_{W_{g,j}} = \begin{cases} 
(4g+2,0;(1,2),(g,2g+1),(1,4g+2)), & \text{if } j = 4g+2, \\
(4g,0;(1,2),(1,4g),(2g-1,4g)), & \text{if } j = 4g,\\
(2g+2,0;(g,g+1),(1,2g+2),(1,2g+2)), & \text{if } j = 2g+2,  \text{ and}\\
(2g+1,0;(1,2g+1),(1,2g+1),(2g-1,2g+1)), & \text{if } j = 2g+1.
\end{cases}$$
\end{lemma}

\noindent We will fix the notation in Lemma~\ref{lem:per_comps} for the remainder of this section. Let $d_i$ denote the boundary curve of $\Sigma_i$ involved in the $1$-compatibility of $F_i$ with $F_{i+1}$, and let $\gamma_i$ represent the isotopy class of $d_i$ in $S_g$ after the compatibility. Let $c^+$ denote the unique integer in $[0,n]$ representing the multiplicative inverse of $c\in \Z_n^{\times}$. We further fix the following notation.
$$ \begin{array}{lcll}
\mu_{1,i}& = & \displaystyle \frac{m_i}{\gcd(m_i,|\tilde{F}_i|)}, & \text{ if } (c_i,n)=(1,|\tilde{F}_i|)_{m_i,\tilde{F}_i}, \\ \\
\mu_{2,i}& = & \displaystyle \frac{m_i}{\gcd(m_i,|\tilde{F}_i|)}{(|\tilde{F}_i|/2-1)^+}, & \text{ if } (c_i,n)=(|\tilde{F}_i|/2-1,|\tilde{F}_i|)_{m_i,\tilde{F}_i}, \text{ and} \\ \\
\mu_{3,i} & = & \displaystyle \frac{m_i}{\gcd(m_i,|\tilde{F}_i|)}{(|\tilde{F}_i|-2)^+}, &  \text{ if } (c_i,n)=(|\tilde{F}_i|-2,|\tilde{F}_i|)_{m_i,\tilde{F}_i}.
\end{array}$$

\noindent By our notation, for each $i$, there exists a unique $z_i \in\{1,2,3\}$ such that  we have $c_i=\mu_{z_i,i}^{-1}\pmod n$, and since $c_i+c_{i+1}\equiv 0\pmod n$, we have $ \mu_{z_i,i}+\mu_{z_{i+1},i+1} \equiv 0\pmod n$. With this notation in place, we have the following lemma, which provides a word $\W(F)$ in Dehn twists that represents the conjugacy class of a star-realizable linear $s$-tuple $F$. 

\begin{lemma}\label{lem:star}
Let $F \in \Mod(S_g)$ be a star-realizable linear $s$-tuple of degree $n$ as in Definition~\ref{defn:star_real}. For all $i$, let $\W(F_i)=(W_{g_i,|\tilde{F}_i |})^{m_i}$. Then:
$$\W(F) = \left( \prod_{i=1}^s \W(F_i) \right)\prod_{i=1}^{s-1}  (T_{\gamma_i})^{-\eta_i},$$ where $\eta_i = \displaystyle \frac{\mu_{z_i,i}+\mu_{z_{i+1},i+1}}{n}$, is conjugate to $F$. 
\end{lemma}
\begin{proof}
Since $\W(F_i)$ commutes with $\W(F_j)$ for $ 1\leq i,j\leq s $, we have $$(\prod_{i=1}^{s} \W(F_i))^n=\prod_{i=1}^{s} (\W(F_i))^n.$$ Since $\W(F_i)=(W_{g_i,|\tilde{F}_i |})^{m_i}$, the fact that $$(c_i,n)\in \{(1, |\tilde{F}_i|)_{m_i,\tilde{F}_i}, (|\tilde{F}_i|/2-1, |\tilde{F}_i|)_{m_i,\tilde{F}_i},(2|\tilde{F}_i|-1, |\tilde{F}_i|)_{m_i,\tilde{F}_i}\}$$ implies that
$$\prod_{i=1}^{s} (\W(F_i))^n=\prod_{i=1}^{s} ((W_{g_i,|\tilde{F}_i |})^{|\tilde{F}_i |})^\frac{m_i}{\gcd(m_i,\tilde{F}_i)}.$$
By Theorem~\ref{thm:gen_starreln}, depending on $|\tilde{F}_i|$ and $D_{\tilde{F}_i}$, $(W_{g_i,|\tilde{F}_i |})^{|\tilde{F}_i|}$ one of:
\[T_{d_1},\, T_{d_1}T_{d_3}, \,T_{d_1}T_{d_2}^{(|\tilde{F}_i|/2-1)^+}, \text{ or } T_{d_1}T_{d_2}^{(|\tilde{F}_i|-2)^+}T_{d_3}.\]
By the definition of $\mu_{z_i,i}$, we have
\[(\prod_{i=1}^{s} \W(F_i))^n=\prod_{i=1}^{s} (\W(F_i))^n=\prod_{i=1}^{s-1} \left(T_{\gamma_i}^{\mu_{z_i,i}+\mu_{z_{i+1},i+1}}\right) \tag{*}\]
As each $T_{\gamma_i}$ commutes with every other Dehn twist appearing in $(*)$ and $\mu_{z_i,i}+\mu_{z_{i+1},i+1}\equiv 0 \pmod n$, we get
$$\W(F)^n =\left(\left( \prod_{i=1}^s \W(F_i) \right)\prod_{i=1}^{s-1}  (T_{\gamma_i})^{-\eta_i}\right)^n= \left( \prod_{i=1}^s \W(F_i) \right)^n\prod_{i=1}^{s-1}  (T_{\gamma_i})^{-\eta_i n}=1,$$
from which the assertion follows.
\end{proof}

\noindent We will describe a algorithm to write a star-realizable linear $s$-tuple $F\in \Mod(S_g)$ as a word in Dehn twists (up to conjugacy). 

\begin{algo}
\label{algo:star_method}
Let $F \in \Mod(S_g)$ be a star-realizable linear $s$-tuple of degree $n$ and genus $g$.
\begin{enumerate}[\textit{Step} 1.]
\item Write $F = (F_1,\ldots,F_s)$ as in Definition~\ref{defn:star_real}. 
\item For each $i$, we set $\W({F_i})=W_{g_i,|\tilde{F}_i|}^{m_i}$, after appropriately relabeling the curves in $\Sigma_i$ in order to ensure consistency with the (assumed) labeling in Theorem~\ref{thm:gen_starreln}.
\item Set $$\W(F) = \left( \prod_{i=1}^s \W(F_i) \right)\prod_{i=1}^{s-1}  (T_{\gamma_i})^{-\eta_i}.$$
\item By Lemma~\ref{lem:star}, $\W(F)$ is the desired representation of $F$ as a word in Dehn twists, up to conjugacy.
\end{enumerate}
\end{algo}

\noindent The method described in Algorithm~\ref{algo:star_method} can be generalized to certain types of $(F,\T)$-tuples. 

\begin{defn}
\label{defn:starreal_FTtuple}
A compatible $(F,\T)$-tuple as in Definition~\ref{defn:comp_FTtuple} is said to be \textit{star-realizable} if the following conditions hold.
\begin{enumerate}[(i)]
\item $v=w=0$.
\item $F$ is star-realizable.
\item For $1\leq q\leq u$, $k_q =1$
\item For $1\leq q\leq u$, suppose the self $1$-compatibility in $\F_{i_{q},j_{q}}$ is along fix points represented by $(c_{i_q},n)$ (in $D_{F_{i_q}}$) and $(n-c_{i_q},n)$ (in $D_{F_{j_q}}$), then 
\begin{gather*}
(c_{i_q},n) \in \{(1, |\tilde{F}_{i_q}|)_{m_{i_q},\tilde{F}_{i_q}}, (|\tilde{F}_{i_q}|/2-1, |\tilde{F}_{i_q}|)_{m_{i_q},\tilde{F}_{i_q}},(|\tilde{F}_{i_q}|-2, |\tilde{F}_{i_q}|)_{m_{i_q},\tilde{F}_{i_q}}\} \\ \text{ and } \\(n-c_{i_q},n) \in \{(1, |\tilde{F}_{j_q}|)_{{m_{j_q},\tilde{F}_{j_q}}}, (|\tilde{F}_{j_q}|/2-1, |\tilde{F}_{j_q}|)_{m_{j_q},\tilde{F}_{j_q}},(|\tilde{F}_{j_q}|-2, |\tilde{F}_{j_q}|)_{m_{j_q},\tilde{F}_{j_q}}\}.
\end{gather*}
\end{enumerate}
\end{defn}

\begin{defn}
\label{defn:star_realgen}
A periodic mapping class $G \in \Mod(S_g)$ is said to be \textit{star-realizable} if there exists a star-realizable compatible $(F,\T)$-tuple $F_\T \in \Mod(S_g)$ and a nonzero integer $m$ such that $G = F_\T^m$.
\end{defn}

\noindent We will now extend Algorithm~\ref{algo:star_method} to this broader class of periodic mapping classes. While doing so, we will retain the notation for $\gamma_i$ and $\eta_i$, for $ 1 \leq i \leq s$ (for $F$) from Algorithm~\ref{algo:star_method}. To further simplify notation, we will denote the additional curves involved in the additional $1$-self compatibilities (of $\F_{\T}$) by $\{\gamma_j\}_{j=s}^{u+s-1}$ and also extend the earlier definition of $\eta_j$ to $s\leq j\leq u+s-1$.

\begin{algo}
\label{algo:star_methodFT}
Let $G\in \Mod(S_g)$ be a star-realizable periodic mapping class.
\begin{enumerate}[\textit{Step} 1.]
\item Write $G = F_{\T}^m$, where $F_{\T}$ is a compatible $(F,\T)$-tuple as in Definition~\ref{defn:comp_FTtuple}. 
\item By Algorithm~\ref{algo:star_method}, we obtain 
$$\W(F) = \left( \prod_{i=1}^s \W(F_i) \right)\prod_{i=1}^{s-1}  (T_{\gamma_i})^{-\eta_i}.$$
\item We set $$\W(F_{\T}) = \W(F) \prod_{i=s}^{u+s-1} (T_{\gamma_i})^{-\eta_i}.$$
\item By the same arguments from Lemma~\ref{lem:star}, $(\W(F_\T))^m$ is the desired representation of $G$ as a word in Dehn twists (after an appropriate relabeling of curves to ensure consistency with Theorem~\ref{thm:gen_starreln}).
\end{enumerate}
\end{algo}

\noindent We will now give three examples to demonstrate the application of Algorithms~\ref{algo:star_method} and \ref{algo:star_methodFT}.

\begin{exmp}
Consider an $F\in\Mod(S_7)$ with $$D_{F}=(6,0;((1,2),2),(1,3),(2,3),(1,6),(5,6)).$$ Then $F$ is a star-realizable linear $3$-tuple $(F_1,F_2,F_3)$, where
$$\begin{array}{rcll}
 D_{F_1} & = & (6,0;((1,2),2),(1,6),(5,6)) & \text{ with } g( D_{F_1}) =  3 \\ 
 D_{F_2} & = & (6,0;(1,3),(5,6),(5,6)) & \text{ with } g( D_{F_2}) =  2, \text{ and } \\ 
 D_{F_3} & = & (6,0;(2,3),(1,6),(1,6)) & \text{ with } g( D_{F_3}) =  2.
\end{array}$$
	\begin{figure}[H]
	\includegraphics[width=55ex]{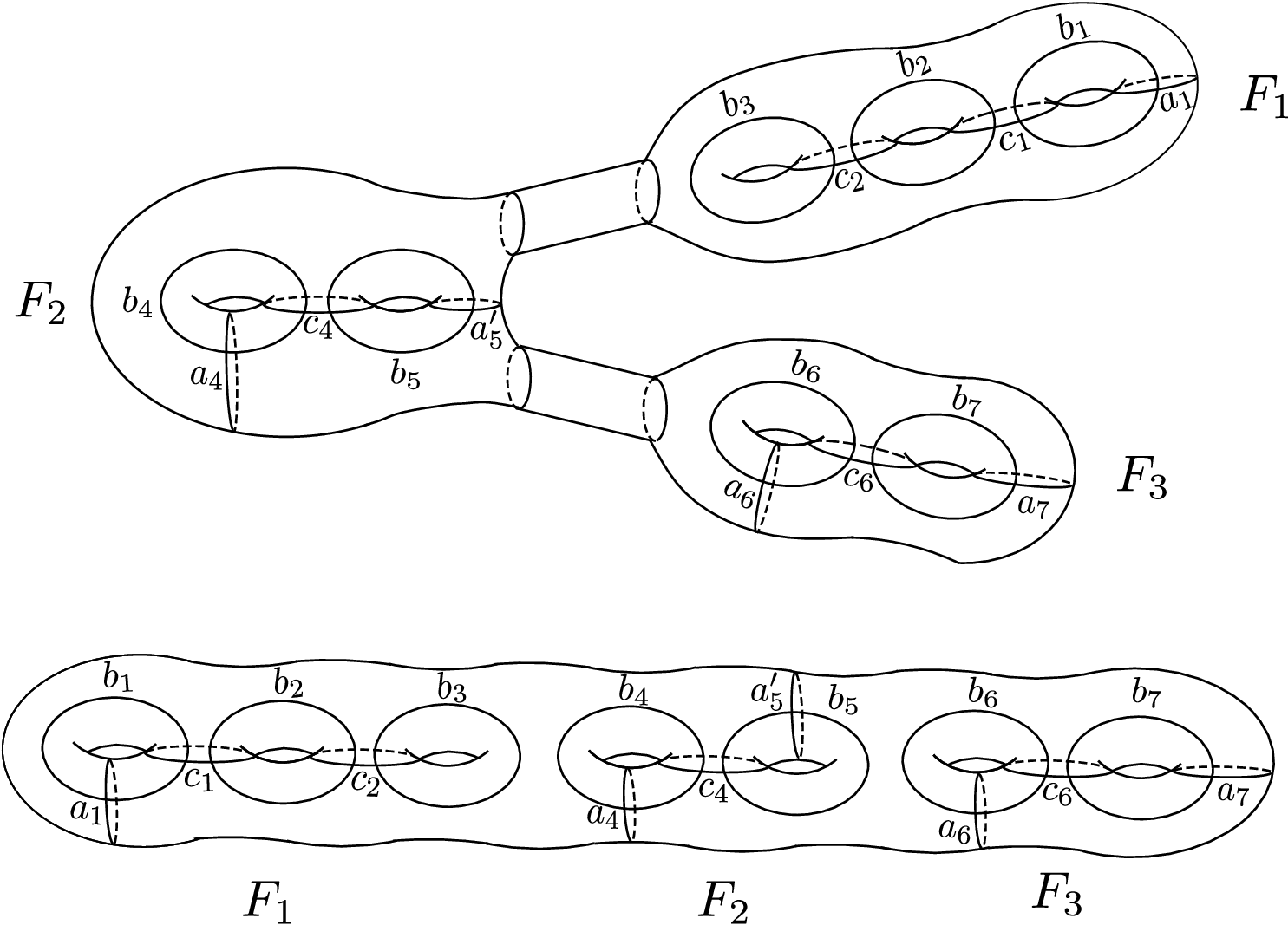}
	\caption{The curves involved in the factorization of $F$.}
	\label{fig:z6_on_s7}
\end{figure}
 Note that the $1$-compatibility of $F_1$ with $F_2$ is along fixed points represented by the pairs $(1,6)$ (in $D_{F_1}$) and $(5,6)$  (in $D_{F_2}$), while the compatibility of $F_2$ with $F_3$ is along the pairs $(5,6)$ (in $D_{F_2}$) and $(1,6)$  (in $D_{F_3}$) (see Figure~\ref{fig:z6_on_s7}). By Algorithm~\ref{algo:star_method}, we have
$\W(F_1)=W_{3,12}^2, \W(F_2)=W_{2,6}^5, \W(F_3)=W_{2,6}$,  and  $\eta_1=1=\eta_2.$ Therefore, we have
$$\W(F)=(T_{a_1}^2T_{b_1}T_{c_1}T_{b_2}T_{c_2}T_{b_3})^2(T_{a_4}T_{b_4}T_{c_4}T_{b_5}T_{a_5'})^5(T_{a_6}T_{b_6}T_{c_6}T_{b_7}T_{a_7})(T_{\gamma_1}T_{\gamma_2})^{-1}.$$
\end{exmp}

\begin{exmp}
Consider a periodic mapping class $G\in\Mod(S_g)$ with $D_{G}=(2g-2,1;(1,2),(1,2))$. Then $G$ is a star-realizable mapping class $F_\T$, where $\T=(1,1,0,0)$ and $F = (F_1)$ with $D_{F_1}=(2g-2,0;(1,2),(1,2),(1,2g-2),(2g-3,2g-2))$ and $g_1 = g(D_{F_1}) = g-1$. Note that the self $1$-compatibility of $F$ is along a pair of fixed points of the $\langle \F \rangle$-action represented by the pairs $(1,2g-2)$ and $(2g-3,2g-2)$ (in $D_{F_1}$). By Algorithm~\ref{algo:star_methodFT}, we have $\W(F_1)=W_{g_1,4g_1}^2,$ and so $$\W(G)=(T_{a_2}\prod_{i=1}^{g-1}(T_{c_i}T_{b_{i+1}}))^{2}T_{a_1}^{-1},$$ where we have relabeled $\gamma_1$ as $a_1$, and $a_1'$ as $c_1$, so as to ensure consistency with Theorem~\ref{thm:gen_starreln}.
\end{exmp}
	
\begin{exmp}
\label{exmp:order7_permadd}
Consider a periodic mapping class $F\in \Mod(S_{10})$ with $D_F=(7,1;(1,7),(3,7),(3,7))$.  Then $F$ is star-realizable linear $2$-tuple $(F_1,F_2)$, where $D_{F_1}= (7,1;(3,7),(4,7))$ with $g_1 = g(D_{F_1}) = 7$ and $D_{F_2}=(7,0;(1,7),(3,7),(3,7))$ with $g_2 = g(D_{F_2}) = 3$. Note that $F_1$ is a rotational mapping class, and the $1$-compatibility of $F_1$ with $F_2$ is along fixed points represented by the pairs $(4,7)$ (in $D_{F_1}$) and $(3,7)$  (in $D_{F_2}$)(see Figure~\ref{fig:z7_on_s10}). Following Algorithm~\ref{algo:star_method}, we have 
$\W(F_1)=W_{7,28}^{8}, \,\W(F_2)=W_{3,7}^5, \text{ and }\eta_1=1.$ Consequently, 
$$\W(F)=(T_{a_1}^2\prod_{i=1}^{6}(T_{b_i}T_{c_i})T_{b_7})^8(T_{a_{10}}^2 T_{b_{10}}T_{c_9}T_{b_9}T_{c_8}T_{b_8}T_{a_8})^5T_{\gamma_1}^{-1},$$
where $\gamma_1$ is the separating curve involved in the $1$-compatibility $F_1$ with $F_2$.
\end{exmp}

\begin{figure}[H]
	\includegraphics[width=65ex]{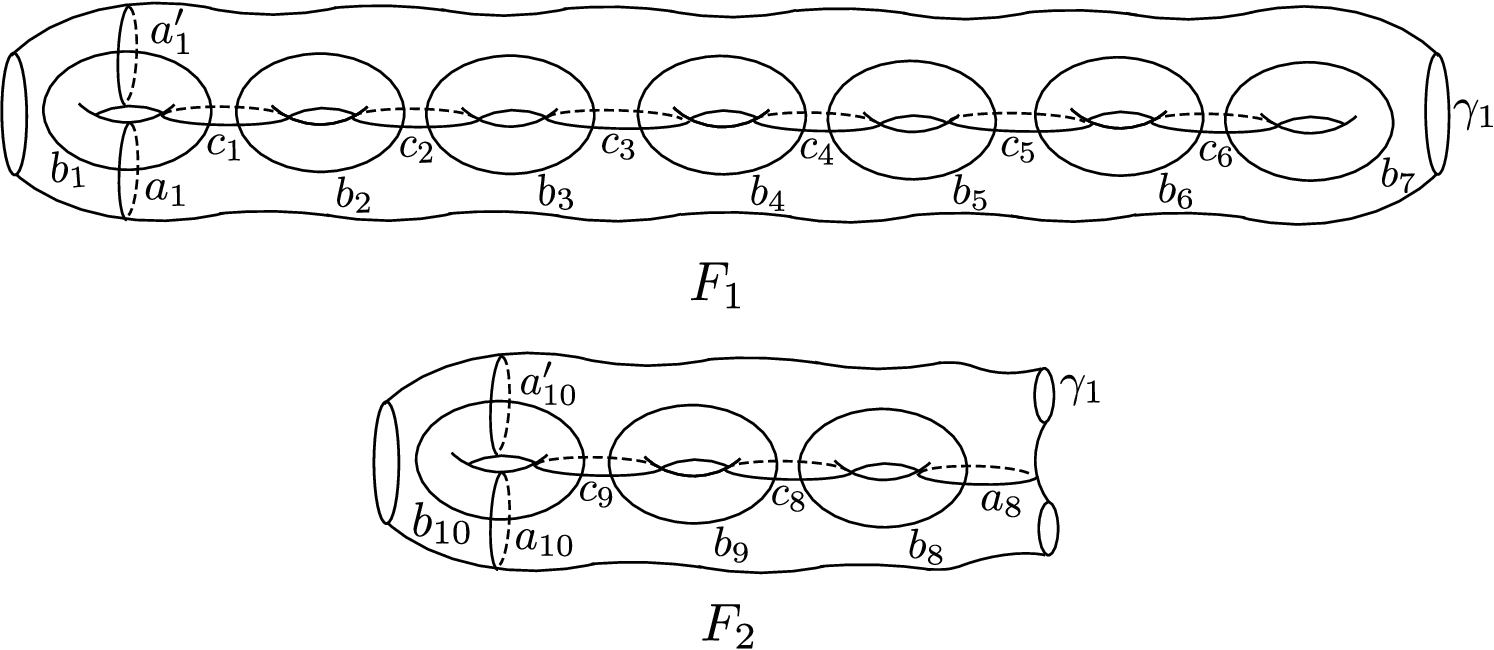}
	\caption{The curves involved in the factorization of $F$.}
	\label{fig:z7_on_s10}
\end{figure}

\noindent In general, the addition of a $g'$-permutation component to a periodic mapping class $F_2 \in \Mod(S_g)$ of order $n$ (as in Construction~\ref{cons:perm_add}) can also be viewed as $1$-compatibility of $F_2$ with the rotational mapping class $F_1 \in \Mod(S_{ng'})$ with $D_{F_1} = (n,g'; (1,n), (n-1,n))$. Note that this compatibility is along fixed points represented by $(n-1,n)$ (in $D_{F_1}$) and $(1,n)$ (in $D_{F_2}$). Moreover, it is not hard to see that $\W(F_1) = W_{ng',4ng'}^{4g'}$, when $n$ is odd. Thus, the ideas in Example~\ref{exmp:order7_permadd} easily generalize to yield the following. 

\begin{prop}
	Let $F_2 \in \Mod(S_g)$ be a periodic star-realizable mapping class of odd order with $D_{F_2}=(n,g_0;(c_1,n),(c_2,n_2),\ldots,(c_r,n_r))$.  Let $F$ be obtained through the addition of a $1$-permutation component to $F_2$. Then viewing $F$ as a $1$-compatible pair $(F_1,F_2)$ along a separating curve $\gamma$ so that $S_g = S_{g(D_{F_2})} \#_{\gamma} S_n$, where $D_{F_1} = (n,1; (c_1,n), (n-c_1,n))$, we have 
	$$\W(F)=\W(F_2)W_{n,4n}^{4c^{+}_1}T_{\gamma}^{-\eta},$$ where $\eta$ is defined along the same lines as in Lemma~\ref{lem:star}.
\end{prop}

\section{Symplectic method}\label{sec:symp_method}
Let $F \in \Mod(S_g)$ be of order $n$. In this section, we give a method by which one can use $\Psi(F)$ for finding a representation of $F$ as a word $\W(F)$, up to conjugacy. (Here, we compute $\Psi(F)$ using Theorem~\ref{thm:rep_irrtype1} and Remark~\ref{rem:irr_symp_blocks}.) In other words, we have to find a suitable candidate for $\W(F)$ in $$\M_F := \{G \in \Mod(S_g):\Psi(G) \text{ is conjugate to } \Psi(F)\}.$$ 

Let $\Psi_m$ denote the composition of $\Psi$ with the canonical projection $\text{Sp}(2g,\Z) \to \text{Sp}(2g,\Z_m)$. It is well known that $\ker \Psi_m$ (also known as the level-$m$ subgroup $\Mod(S_g)[m]$) is torsion-free for $m \geq 3$ (see~\cite[Theorem 6.9]{FM}). Considering that the conjugacy class of $\Psi(F)$ can be infinite in $\text{Sp}(2g,\Z)$, for computational purposes, we consider the set 
$$\TM_F := \{G \in \Mod(S_g):\Psi_3(G) \text{ is conjugate to } \Psi_3(F)\}$$ 
in place of $\M(F)$. The key idea behind our method is to provide a systematic procedure for carefully and efficiently sifting through the elements in set $\TM_F$ to find a suitable candidate for $\W(F)$. 

\subsection{Structured searching for $\W(F)$} To standardize our procedure, we consider the Lickorish~\cite{WBRL} generating set $\L_g$ for $\Mod(S_g)$ and assume that each element in $\M_F$ is a word in $\L_g$. To fix notation, let $$\L_g = \{T_{a_1},T_{b_1},T_{c_1},T_{b_2},T_{a_2},T_{c_2}, \ldots,T_{c_{g-1}},T_{b_g},T_{a_g}\}$$ with the $a_i$, the $b_i$ and the $c_i$ are indicated in Figure~\ref{fig:2_lick_gens} below. 
\begin{figure}[H]
        \labellist
		\tiny
		\pinlabel $a_1$ at 109 50
		\pinlabel $b_1$ at 132 185
		\pinlabel $a_g$ at 492 50
		\pinlabel $b_g$ at 525 185
		\pinlabel $c_{1}$ at 219 90
		\pinlabel $c_{g-1}$ at 430 85
		\endlabellist
	\includegraphics[width=40ex]{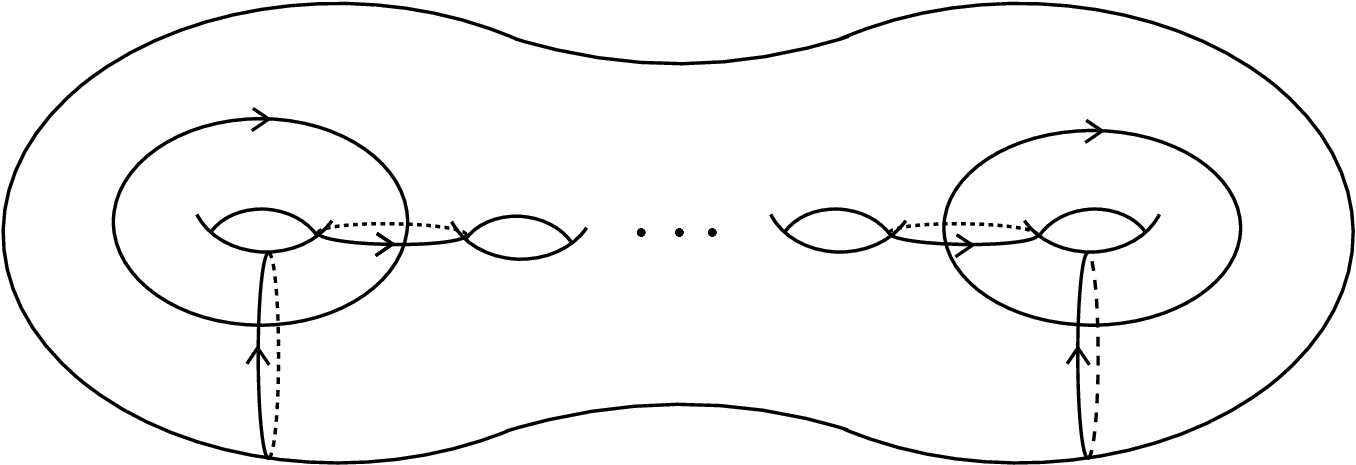}
	\caption{The curves in $S_g$ involved in the Lickorish twists.}
	\label{fig:2_lick_gens}
\end{figure}
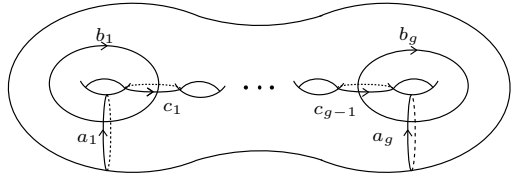 
\noindent 

\noindent In order to make our search for $\W(F)$ in $\TM_F$ more efficient, we first have to ensure the implementation of a well-structured search process.  For achieving this, we introduce the notion of the \textit{depth} of a word. Let $\W$ be a reduced word in $\L_g$, and let $n_i$ be the number of times the $i^{th}$ generator in $\L_g$ appears in $\W$. Then the \textit{depth $d(\W)$} of the word $\W$ is defined by $d(\W) =\max \{n_i: 1 \leq i \leq 3g-1\}$. For example, for the word $\W = T_{a_1}^5 T_{a_2}^4 T_{b_2} T_{a_2}^2 T_{b_1}$, $d(\W) = 2$. Further, we denote the largest power (in absolute value) of a Dehn twist in $\L_g$ appearing in a word $W$ by $p(\W)$, and fix the notation $\TM_{F}^{i,j}:= \{W \in \TM_F: d(\W) =i \text{ and } p(\W) = j\}.$ Thus, we will begin our search for $\W(F)$ in $\TM_{F}^{1,1}$, and then gradually broaden our search in an incremental manner to $\TM_F^{i,j}$ for $i,j >1$.

\subsection{Discarding redundant words} To begin with, we apply the basic property that Dehn twists about isotopically disjoint curves in $S_g$ commute, we would like to discard redundant variants of words that are equivalent up to commutativity of the twists in $\L_g$. For this reason, we assign numbers $1$ through $3g-1$ for the Dehn twists in appearing (in sequence) in $\L_g$. A permutation $\sigma$ of $\{1,\ldots 3g-1\}$ is said to be \textit{good} if for $1 \leq i \leq 3g-1$, either $\sigma(i+1) - \sigma(i) \leq 1$ or $(\sigma(i), \sigma(i+1)) = (3k-2, 3k) $ for some \(k\). Thus, in our process, we will filter out many (redundant) words in $\TM_F$ by considering only words arise as good permutations of the (powers of the) Dehn twists appearing in $\L_g$. We will further discard several non-periodic words in $\TM_F$ by applying the Penner's construction~\cite{RP} of pseudo-Anosov mapping classes.
\begin{theorem}
\label{thm:pA_recipe}
Let $\C = \{\alpha_1,\ldots,\alpha_n\}$ and $\D = \{\alpha_{n+1},\ldots,\alpha_{n+m}\}$ be multicurves in $S_g$ that together fill $S_g$. Then any product of positive powers of the $T_{\alpha_i}$, for $i=1,\ldots,n$ and negative powers of the $T_{\alpha_{n+j}}$, for $j=1,\ldots,m$, where each $\alpha_i$ and each $\alpha_{n+j}$ appears at least once, is pseudo-Anosov.
\end{theorem}

\subsection{Searching for periodics} Let $i(\alpha, \beta)$ be the geometric intersection number of essential simple closed curves $\alpha, \beta$ in $S_g$. In order to identify the periodics in $\TM_F$, we will (in general) use the well-known Bestvina-Handel algorithm~\cite{BH}, which provides an effective way of identifying them. This brings us to the following remark

\begin{rem}
\label{rem:iden_rec_period}
The Bestvina-Handel algorithm also allows us to effectively distinguish between reducible mapping classes of finite and infinite order. When the algorithm detects a reduction, we can conclude whether the mapping class is periodic (or not) by observing whether the growth rate of the transition matrix associated with the mapping class is $1$ (or otherwise). 
\end{rem}

\begin{rem}
\label{rem:iden_irrperiod}
When $F$ is irreducible, the elements $\TM_F$ are either irreducible periodics or pseudo-Anosovs. In this context, we have observed that it is easier to identify the periodics by simply determining whether the orbits (under $F$) of certain appropriately chosen curves are finite. To this end, a software named Teruaki for Mathematica 7 (or TKM7) by Sakasai-Suzuki~\cite{SST} designed for the visualization of actions of Dehn twists on curves (in $S_g$) really comes in handy.
\end{rem}

\noindent We are now in a position to describe our method. 

\begin{algo}[Symplectic method] 
\label{algo:symp_method}
Let $F \in \Mod(S_g)$ be of order $n$. 
\begin{enumerate}[\textit{Step} 1.] 
\item Compute $\Psi(F)$ (up to conjugacy) using Theorem~\ref{thm:rep_irrtype1} and Remark~\ref{rem:irr_symp_blocks}.
\item Set $i=1$, $j=1$, and $flag =0$
\item Repeat Steps 4-5 until $flag =1$. 
\item Repeat Steps 4a - 4f, while $j \leq n$. 
\begin{enumerate}[\textit{Step 4}a.]
\item Compute the elements in $\TM_{F}^{i,j}$ up to good permutations.
\item Discard the words in $\TM_{F}^{i,j}$ that are pseudo-Anosovs using Theorem~\ref{thm:pA_recipe} (or Remark~\ref{rem:iden_irrperiod} when $F$ is irreducible).
\item If $|\TM_{F}^{i,j}| >0$, repeat Steps 4d - 4e, for each $W \in \TM_{F}^{i,j}$. Else, proceed to Step 4f. 
\item Apply the Bestvina-Handel algorithm to determine the mapping class type of $W$. 
\item If $W$ is periodic, then set $\W(F) = W$, $flag =1$, and then proceed to Step 5. Else, set $\TM_{F}^{i,j} = \TM_{F}^{i,j} \setminus \{W\}$ and proceed to Step 4c. 
\item Set $j=j+1$ and proceed to Step 4. 
\end{enumerate}
\item If $flag =0$, set $i = i+1$ and proceed to Step 3. Else, proceed to Step 6.
\item $\W(F)$ is the desired representation of $F$ as a word in Dehn twists, up to conjugacy.
\end{enumerate}
\end{algo}

\begin{exmp}
\label{eg:symp_rep_order9_map}
Consider the irreducible Type 1 mapping class \(F \in \Mod(S_3)\) with \(D_F = (9,0;(1,3),(1,9),(5,9))\). Clearly, $F$ is neither rotational nor star-realizable, so we will use Algorithm~\ref{algo:symp_method} to find $\W(F)$. We will follow the notation from Subsections~\ref{sec:pri_into_irred}-\ref{sec:per_symp_rep}. By virtue of Theorem~\ref{res:1}, $F$ is realized as a rotation of the polygon $\P_F$ by $4\pi/9$, as shown in Figure~\ref{fig:polys} below, with 
\begin{gather*} 
L(\P_F) = \{ a_1,a_2,a_3,a_4,a_5,a_6,a_7,a_8,a_9\} \text{ and } \\ W(\P_F) = a_1a_2a_3a_4a_5a_6a_2^{-1}a_7a_4^{-1}a_8a_6^{-1}a_9a_7^{-1}a_1^{-1}a_8^{-1}a_3^{-1}a_9^{-1}a_5^{-1}.
\end{gather*}
Let $\mathcal{P}':=\P_3$ be the standard $12$-gon (realizing the surface $S_3$) with  $$L(\P_3) = \{x_1,y_1,x_2,y_2,x_3,y_3\} \text{ and }W(\P_3) = [x_1, y_1] [x_2, y_2] [x_3,y_3].$$
		\begin{figure}[h]
		\centering
		\begin{subfigure}{.5\textwidth}
			\centering
			\includegraphics[width=.8\textwidth]{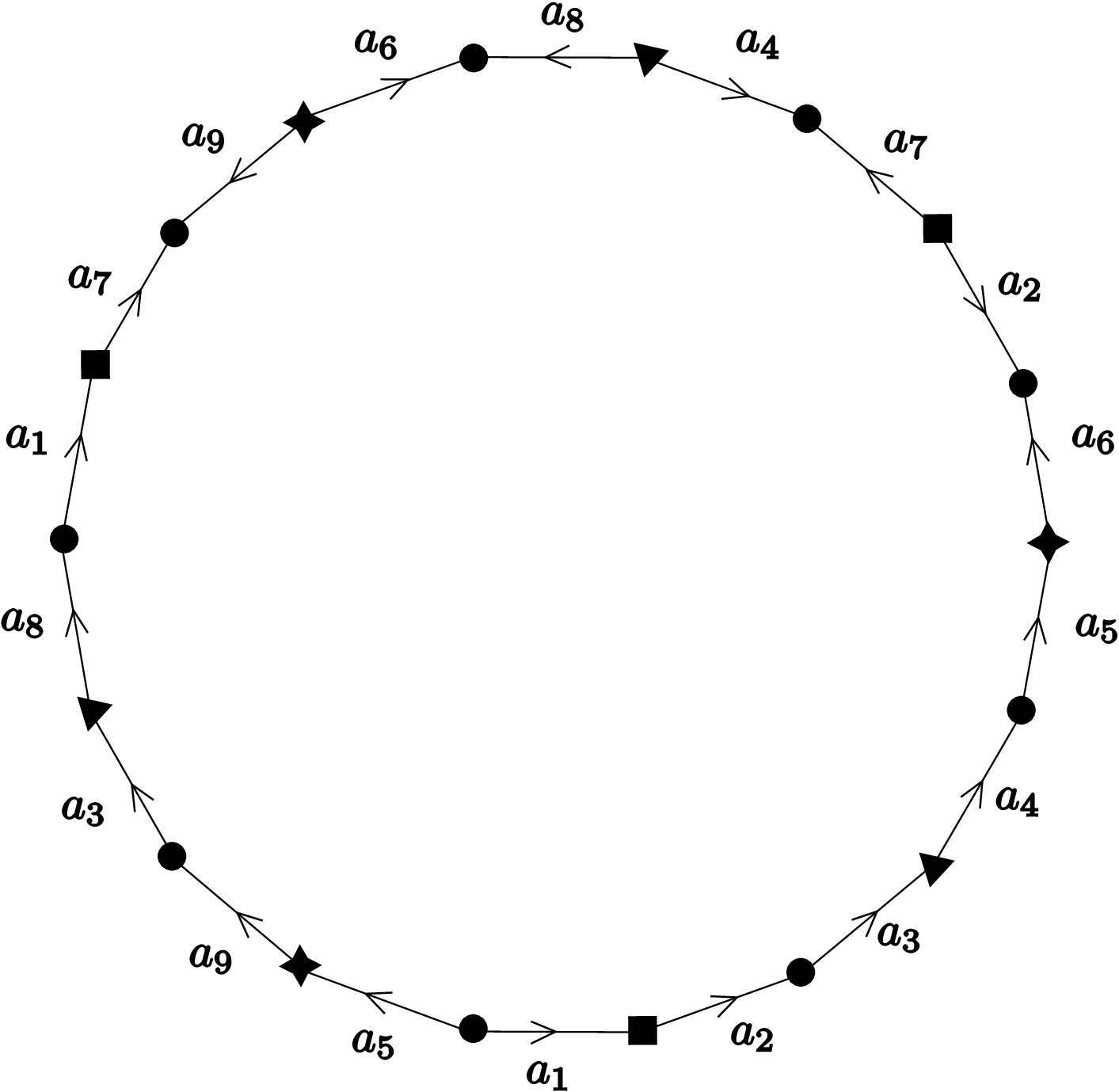}
			\caption{$\mathcal{P}_F$}
			\label{fig:3_g309}
		\end{subfigure}%
		\begin{subfigure}{.5\textwidth}
			\centering
			\includegraphics[width=.8\textwidth]{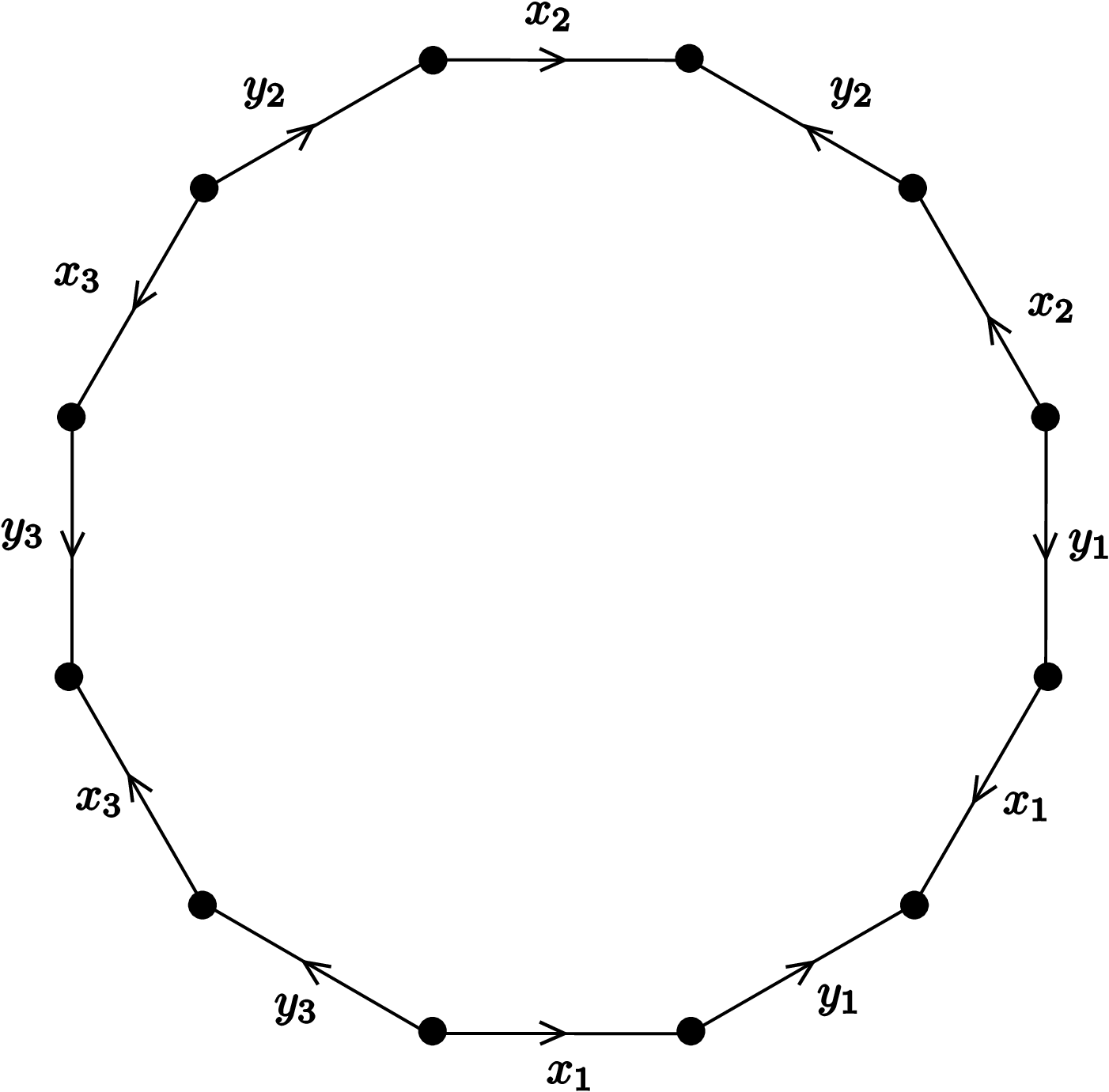}
			\caption{$\mathcal{P}_3$}
			\label{fig:4_g3stdpolyg}
		\end{subfigure}
		\caption{The polygons $\P_F$ and $\P_3$.}
		\label{fig:polys}
	\end{figure}
Denoting $\phi = \phi_{\P_F}$ and $f = f_{\P_F}$, we obtain $\varphi  = f^{-1} \phi f$ (as in Theorem~\ref{thm:rep_irrtype1}) in the following manner. 
$$
\small
\begin{array}{lclclcl}
		\left[ x_1 \right] & \xrightarrow{f} & \left[ b_1 \right]+\left[ b_3 \right] & \xrightarrow{\phi} &  \left[ b_2 \right]-\left[ b_3 \right] +\left[b_4\right] &\xrightarrow{f^{-1}} & \left[y_1\right] - \left[x_1\right] +\left[y_3\right]-\left[x_3\right] \\
	
		\left[ y_1 \right] &\xrightarrow{f}& \left[ b_1 \right]+\left[ b_2 \right] &\xrightarrow{\phi} &  -\left[ b_1 \right]-\left[ b_3 \right]-\left[ b_6 \right] & \xrightarrow{f^{-1}} &  - \left[ x_1 \right] +\left[ y_2 \right] - \left[x_2\right] \\
		
		\left[ x_2 \right] &  \xrightarrow{f} &\left[ b_4 \right] + \left[b_5\right] + \left[b_6\right]+\left[b_2\right] & \xrightarrow{\phi} & \left[ b_3 \right] - \left[b_4\right] + \left[b_5\right] - \left[b_6\right] & \xrightarrow{f^{-1}} & -\left[y_1\right]+\left[x_1\right]+2\left[y_2\right] \\ 
		
		& & & &&& -\left[x_2\right]-2\left[y_3\right] + 2\left[x_3\right]\\ 
		
		\left[ y_2 \right] & \xrightarrow{f} & \left[ b_2 \right] + \left[b_4\right] + \left[b_5\right] & \xrightarrow{\phi} & -\left[ b_4 \right] - \left[b_6\right] & \xrightarrow{f^{-1}} & \left[y_2 \right] - \left[x_2\right]-\left[y_3\right] + \left[x_3\right] \\

		\left[ x_3 \right] & \xrightarrow{f} & -\left[ b_4 \right] - \left[b_5\right] & \xrightarrow{\phi} &  \left[ b_4 \right] - \left[b_1 \right] & \xrightarrow{f^{-1}} & -\left[y_1\right] +\left[y_2\right]+\left[y_3\right]\\

		\left[ y_3 \right] & \xrightarrow{f} & -\left[ b_5 \right] & \xrightarrow{\phi} & \left[ b_4 \right] &  \xrightarrow{f^{-1}} & \left[ y_3 \right] - \left[x_3\right] \\
	 
\end{array}$$ Here, $[b_1] = [a_8^{-1} a_3^{-1}],[b_2] = [a_9^{-1} a_5^{-1}],[b_3] = [a_2^{-1} a_1^{-1}],[b_4] = [a_6^{-1} a_9],[b_5] = [a_7^{-1} a_2], \text{ and }[b_6] = [a_4^{-1} a_8]$.
Thus, the matrix $M_{\varphi}$ representing the conjugacy class of $\Psi(F)$ in $\text{Sp}(2g,\Z)$ is given by
	\[M_{\varphi} =
	\left[ \begin{array}{rrrrrr}
	-1 & -1 & 1 &0 & 0 &0 \\
	1 &0 & -1 & 0 & -1 & 0\\
	0 & -1 & -1 & -1 & 0 & 0\\
	0 & 1 & 2 & 1 & 1 & 0\\
	-1 & 0 & 2 & 1 & 0 & -1\\
	1 & 0 & -2 & -1 & 1 & 1 \\
	\end{array}
	\right],   \text{ and }
		 \Psi_3(F) =
	\begin{bmatrix}
	2 & 2 & 1 & 0 & 0 & 0 \\
	1 & 0 & 2 & 0 & 2 & 0\\
	0 & 2 & 2 & 2 & 0 & 0\\
	0 & 1 & 2 & 1 & 1 & 0\\
	2 & 0 & 2 & 1 & 0 & 2\\
	1 & 0 & 1 & 2 & 1 & 1
	\end{bmatrix}
\]
Following Algorithm~\ref{algo:symp_method}, we begin our search for $\W(F)$ in $\TM_{F}^{1,1}$. As it turns out, even after considering only the words among the good permutations that are not Penner-type pseudo-Anosovs, we were still left with numerous (at least 150) possible candidates for $\W(F)$ in $\TM_{F}^{1,1}$. For brevity, we will demonstrate the algorithm on a small subcollection of words
\begin{gather*}
\TM_{F}^{1,1} (\sigma) =\{ T_{b_3}^{-1} T_{a_3}^{-1} T_{c_2}^{-1} T_{a_2}^{-1}T_{b_2}^{-1} T_{c_1}^{-1} T_{b_1}^{-1},
T_{b_3}^{-1}T_{c_2}^{-1} T_{a_2}^{-1}T_{b_2}^{-1} T_{c_1}^{-1} T_{b_1}^{-1}T_{a_1}^{-1}, \\
T_{b_3} T_{a_3} T_{c_2} T_{a_2} T_{b_2} T_{c_1} T_{b_1} ,
T_{b_3} T_{c_2} T_{a_2} T_{b_2} T_{c_1} T_{b_1} T_{a_1} \}
\end{gather*}
corresponding to the permutation of $\sigma  = (1\,7\,2\,8)(3\,6)(4\,5)$ on $\L_3$. Using Remark~\ref{rem:iden_irrperiod} (and the Teruaki software ~\cite{SST}), we can easily deduce that all words in $\TM_{F}^{1,1}(\sigma)$ are finite order. Thus, we may choose $\W(F)=T_{b_3} T_{c_2} T_{a_2} T_{b_2} T_{c_1} T_{b_1} T_{a_1}$.
\end{exmp}

Note that the symplectic method can be applied to any periodic mapping class. However, as the method is computationally intense, we recommend its application only for non-rotational periodics that are neither star-realizable nor chain-realizable. It is apparent that while the earlier methods were more restrictive in terms of their applicability, they work quite efficiently for the specific families of periodic mapping classes they were designed for.

\section{Applications}\label{sec:mthd_appls}

\subsection{Factoring periodic elements in $\Mod(S_3)$}\label{subs:words_mods3} Using Algorithms~\ref{alg:inv}, \ref{algo:chain_method}, \ref{algo:star_methodFT}, and~\ref{algo:symp_method}, in Table~\ref{tab:genus3_words} below, we provide a word $\W(F)$ (in Dehn twists) representing the conjugacy class of each periodic mapping class $F \in \Mod(S_3)$. (A similar list has also been derived in~\cite{SH}.)

\begin{table}[htbp]
	\small
	\begin{center}
		\scalebox{0.85}{
			\begin{tabular}{|c|c|c|c| }
				\hline 
				$|F|$ & $D_F$ & $\W(F)$ & Algorithm\\
				\hline 
				14 & $(14,0;(1,2),(3,7),(1,14))$ & $(T_{a_1}T_{b_1}T_{c_1}T_{b_2}T_{c_2}T_{b_3})$ & \ref{algo:chain_method} \\
				14 & $(14,0;(1,2),(2,7),(3,14))$ & $(T_{a_1}T_{b_1}T_{c_1}T_{b_2}T_{c_2}T_{b_3})^5$ & \ref{algo:chain_method} \\
				14 & $(14,0;(1,2),(1,7),(5,14))$ & $(T_{a_1}T_{b_1}T_{c_1}T_{b_2}T_{c_2}T_{b_3})^3$ & \ref{algo:chain_method} \\
				14 & $(14,0;(1,2),(6,7),(9,14))$ & $(T_{a_1}T_{b_1}T_{c_1}T_{b_2}T_{c_2}T_{b_3})^{11}$ & \ref{algo:chain_method} \\
				14 & $(14,0;(1,2),(5,7),(11,14))$ & $(T_{a_1}T_{b_1}T_{c_1}T_{b_2}T_{c_2}T_{b_3})^9$ & \ref{algo:chain_method} \\
				14 & $(14,0;(1,2),(4,7),(13,14))$ & $(T_{a_1}T_{b_1}T_{c_1}T_{b_2}T_{c_2}T_{b_3})^{13}$ & \ref{algo:chain_method} \\
				12 & $(12,0;(1,2),(5,12),(1,12))$ & $(T_{a_1}^2T_{b_1}T_{c_1}T_{b_2}T_{c_2}T_{b_3})$ & \ref{algo:chain_method} \\
				12 & $(12,0;(1,2),(7,12),(11,12))$ & $(T_{a_1}^2T_{b_1}T_{c_1}T_{b_2}T_{c_2}T_{b_3})^7$ & \ref{algo:chain_method} \\
				12 & $(12,0;(2,3),(1,4),(1,12))$ & $(T_{b_3} T_{c_2} T_{a_2} T_{b_2} T_{c_1} T_{b_1})$  & \ref{algo:symp_method} \\
				12 & $(12,0;(1,3),(1,4),(5,12))$ & $(T_{b_3} T_{c_2} T_{a_2} T_{b_2} T_{c_1} T_{b_1})^5$  & \ref{algo:symp_method} \\
				12 & $(12,0;(2,3),(3,4),(7,12))$ & $(T_{b_3} T_{c_2} T_{a_2} T_{b_2} T_{c_1} T_{b_1})^7$  & \ref{algo:symp_method} \\
				12 & $(12,0;(1,3),(3,4),(11,12))$ & $(T_{b_3} T_{c_2} T_{a_2} T_{b_2} T_{c_1} T_{b_1})^{11}$  & \ref{algo:symp_method} \\
				9 & $(9,0;(1,3),(5,9),(1,9))$ & $(T_{b_3} T_{c_2} T_{a_2} T_{b_2} T_{c_1} T_{b_1} T_{a_1})$ & \ref{algo:symp_method} \\
				9 & $(9,0;(2,3),(1,9),(2,9))$ & $(T_{b_3} T_{c_2} T_{a_2} T_{b_2} T_{c_1} T_{b_1} T_{a_1})^5$ & \ref{algo:symp_method} \\
				9 & $(9,0;(1,3),(2,9),(4,9))$ & $(T_{b_3} T_{c_2} T_{a_2} T_{b_2} T_{c_1} T_{b_1} T_{a_1})^7$ & \ref{algo:symp_method} \\
				9 & $(9,0;(2,3),(7,9),(5,9))$ & $(T_{b_3} T_{c_2} T_{a_2} T_{b_2} T_{c_1} T_{b_1} T_{a_1})^2$ & \ref{algo:symp_method} \\
				9 & $(9,0;(1,3),(8,9),(7,9))$ & $(T_{b_3} T_{c_2} T_{a_2} T_{b_2} T_{c_1} T_{b_1} T_{a_1})^4$ & \ref{algo:symp_method} \\
				9 & $(9,0;(2,3),(4,9),(8,9))$ & $(T_{b_3} T_{c_2} T_{a_2} T_{b_2} T_{c_1} T_{b_1} T_{a_1})^8$ & \ref{algo:symp_method} \\
				8 & $(8,0;(3,4),(1,8),(1,8))$&$(T_{a_1}T_{b_1}T_{c_1}T_{b_2}T_{c_2}T_{b_3}T_{a_3})$ & \ref{algo:chain_method} \\
				8 & $(8,0;(1,4),(3,8),(3,8))$&$(T_{a_1}T_{b_1}T_{c_1}T_{b_2}T_{c_2}T_{b_3}T_{a_3})^3$ & \ref{algo:chain_method} \\
				8 & $(8,0;(3,4),(5,8),(5,8))$&$(T_{a_1}T_{b_1}T_{c_1}T_{b_2}T_{c_2}T_{b_3}T_{a_3})^5$ & \ref{algo:chain_method} \\
				8 & $(8,0;(1,4),(7,8),(7,8))$&$(T_{a_1}T_{b_1}T_{c_1}T_{b_2}T_{c_2}T_{b_3}T_{a_3})^7$ & \ref{algo:chain_method} \\
				8 & $(8,0;(1,4),(5,8),(1,8))$& $(T_{a_1}T_{c_1}T_{b_3}T_{a_3}T_{c_2}T_{b_2}T_{b_1}) $ & \ref{algo:symp_method} \\
				8 & $(8,0;(3,4),(7,8),(3,8))$& $(T_{a_1}T_{c_1}T_{b_3}T_{a_3}T_{c_2}T_{b_2}T_{b_1})^3 $ & \ref{algo:symp_method} \\
				7 & $(7,0;(5,7),(1,7),(1,7))$&$(T_{a_1}^2T_{b_1}T_{c_1}T_{b_2}T_{c_2}T_{b_3}T_{a_3})$ & \ref{algo:chain_method} \\
				7 & $(7,0;(3,7),(2,7),(2,7))$&$(T_{a_1}^2T_{b_1}T_{c_1}T_{b_2}T_{c_2}T_{b_3}T_{a_3})^4$ & \ref{algo:chain_method} \\
				7 & $(7,0;(1,7),(3,7),(3,7))$&$(T_{a_1}^2T_{b_1}T_{c_1}T_{b_2}T_{c_2}T_{b_3}T_{a_3})^5$ & \ref{algo:chain_method} \\
				7 & $(7,0;(6,7),(4,7),(4,7))$&$(T_{a_1}^2T_{b_1}T_{c_1}T_{b_2}T_{c_2}T_{b_3}T_{a_3})^2$ & \ref{algo:chain_method} \\
				7 & $(7,0;(4,7),(5,7),(5,7))$&$(T_{a_1}^2T_{b_1}T_{c_1}T_{b_2}T_{c_2}T_{b_3}T_{a_3})^3$ & \ref{algo:chain_method} \\
				7 & $(7,0;(2,7),(6,7),(6,7))$&$(T_{a_1}^2T_{b_1}T_{c_1}T_{b_2}T_{c_2}T_{b_3}T_{a_3})^6$ & \ref{algo:chain_method} \\
				7 & $(7,0;(4,7),(2,7),(1,7))$& $ (T_{b_1}T_{b_2}T_{c_2}T_{a_3}T_{b_3}T_{c_1}T_{a_1}^2)^{-1} $  & \ref{algo:symp_method} \\
				7 & $(7,0;(3,7),(5,7),(6,7))$& $ (T_{b_1}T_{b_2}T_{c_2}T_{a_3}T_{b_3}T_{c_1}T_{a_1}^2)^{-6} $  & \ref{algo:symp_method} \\
				6 & $(6,0;(1,2),(1,2),(1,6),(5,6))$&$(T_{a_1}^2T_{b_1}T_{c_1}T_{b_2}T_{c_2}T_{b_3})^2 $& \ref{algo:chain_method} \\
				6 & $(6,0;(1,2),(2,3),(2,3),(1,6))$& $(T_{a_1}T_{b_1})^{-1}(T_{a_2}T_{b_2}T_{c_2}T_{b_3}T_{a_3})$& \ref{algo:chain_method} \\
				6 & $(6,0;(1,2),(1,3),(1,3),(5,6))$& $(T_{a_1}T_{b_1})(T_{a_2}T_{b_2}T_{c_2}T_{b_3}T_{a_3})^{-1}$& \ref{algo:chain_method} \\
				4 & $(4,1;(1,2),(1,2))$& $(T_{a_2}T_{c_1}T_{b_2}T_{c_2}T_{b_3})^{2}T_{a_1}^{-1}$ & \ref{algo:star_methodFT} \\
				4 & $(4,0;(1,2),(1,2),(1,2),(1,4),(1,4))$& $(T_{a_1}^2T_{b_1})(T_{a_2}T_{a_2'}T_{b_2})^3(T_{a_3}^2T_{b_3})(T_{s_1}T_{s_2})^{-1}$  & \ref{algo:star_method} \\
				4 & $(4,0;(1,2),(1,2),(1,2),(3,4),(3,4))$& $(T_{a_1}^2T_{b_1})^3(T_{a_2}T_{a_2'}T_{b_2})(T_{a_3}^2T_{b_3})^3(T_{s_1}T_{s_2})^{-1}$  & \ref{algo:star_method} \\
				4 & $(4,0;(1,4),(1,4),(3,4),(3,4))$& $(T_{a_1}T_{b_1}T_{c_1}T_{b_2}T_{c_2}T_{b_3}T_{a_3})^2$ & \ref{algo:chain_method} \\
				4 & $(4,0;(1,4),(1,4),(1,4),(1,4))$& $(T_{b_3} T_{c_2} T_{a_2} T_{b_2} T_{c_1} T_{b_1})^3$ & \ref{algo:symp_method} \\
				4 & $(4,0;(3,4),(3,4),(3,4),(3,4))$& $(T_{b_3} T_{c_2} T_{a_2} T_{b_2} T_{c_1} T_{b_1})^9$ & \ref{algo:symp_method} \\
				3 & $(3,0;(1,3),(1,3),(1,3),(1,3),(2,3))$& $(T_{a_1}^3T_{b_1})(T_{a_2}^2T_{a_2'}T_{b_2})^2(T_{a_3}^3T_{b_3})(T_{s_1}T_{s_2})^{-1}$  & \ref{algo:star_method} \\
				3 & $(3,0;(2,3),(2,3),(2,3),(2,3),(1,3))$& $(T_{a_1}^3T_{b_1})^2(T_{a_2}^2T_{a_2'}T_{b_2})(T_{a_3}^3T_{b_3})^2(T_{s_1}T_{s_2})^{-1}$  & \ref{algo:star_method} \\
				3 & $(3,1;(1,3),(2,3))$& $(T_{a_1}^2T_{c_1}T_{b_1})(T_{a_3}^2T_{c_2}T_{b_3})^2(T_{a_2}T_{a_2'})^{-1}$& \ref{algo:star_methodFT} \\
				2 & $(2,2;1)$& $(T_{a_1}T_{b_1}T_{c_1}T_{b_2}T_{c_2})^3T_{a_3}^{-1}$& \ref{algo:star_methodFT} \\
				2 & $(2,1;((1,2),4))$&$(T_{a_2}T_{b_2}T_{x_1}T_{a_1}T_{b_1})^3(T_{a_3}T_{b_3}T_{a_3})^{-2}$ & \ref{alg:inv} \\
				2 & $(2,0;((1,2),8))$& $(T_{a_1}T_{b_1}T_{a_1})^{2}(T_{a_2}T_{b_2}T_{a_2'})^{-2}(T_{a_3}T_{b_3}T_{a_3})^{2}$& \ref{alg:inv} \\
				\hline
			\end{tabular}}
		\end{center}
		\caption{Words (in Dehn twists) representing the conjugacy classes of periodic elements in $\Mod(S_3)$.}
		\label{tab:genus3_words}
	\end{table}

\subsection{Roots of Dehn twists}\label{subs:roots_ofdtwists} For $g \geq 2$, let $c$ be a simple closed curve in $S_g$. A \textit{root of $T_c$ of degree n} is an $F \in \Mod(S_g)$ such that $F^n  = T_c$. 

\subsubsection{Dehn twists about nonseparating curves} When $c$ is nonseparating, Margalit-Schleimer~\cite{MS} gave the first example of such a root of degree $2g-1$ in $\Mod(S_g)$. A complete classification of such roots was obtained~\cite{MK1}, where it was also shown that the Margalit-Scleimer root (of degree $2g-1$) had the largest possible degree in $\Mod(S_g)$. A periodic mapping class $\bar{F} \in \Mod(S_{g-1})$ is said to be \textit{degree n root-realizing} if the $\langle \bar{\F} \rangle$-action on $S_{g-1}$ has two distinguished fixed points where the induced local rotation angles add up to $2\pi/n \pmod{2\pi}$. Given a root-realizing $\bar{F} \in \Mod(S_{g-1})$ of order $n$ with distinguished fixed points $P_1$ and $P_2$, one can remove $\langle \bar{\F} \rangle$-invariant neighborhoods around the $P_i$ and then attach an annulus $A$ with an $(1/n)^{th}$-twist connecting the resulting boundary components to realize a root $F \in \Mod(S_g)$ of a Dehn twist $T_c$ about the non-separating curve $c$ in $A$. Conversely, for $g \geq 2$, given a root $F \in \Mod(S_g)$ of $T_c$  of degree $n$, one can reverse this process to recover a root-realizing periodic mapping class $\bar{F} \in \Mod(S_{g-1})$. Thus, the conjugacy class of a typical root realizing $\bar{F} \in \Mod(S_{g-1})$ that corresponds to the conjugacy class of a root $F \in \Mod(S_g)$ of degree $n$ has the form 
$$D_{\bar{F}}=(n,g_0;(a,n),(b,n), (c_1,n_1),\ldots,(c_{\ell}, n_{\ell})),$$ 
where $a+b \equiv ab \pmod{n}$. Here the pairs $(a,n)$ and $(b,n)$ represent the fixed points (of the $\langle \bar{\F} \rangle$-action) involved in the construction of the root $F$. We will now describe a family of roots that can be represented as words in Dehn twists by using a minor modification of the method described in Algorithm~\ref{algo:star_methodFT}.

\begin{defn}
\label{defn:star_root}
A root $F \in \Mod(S_g)$ of $T_c$ of degree $n$ is said to be \textit{star-realizable} if the following conditions hold. 
\begin{enumerate}[(i)]
\item The root-realizing periodic mapping class $\bar{F} \in \Mod(S_{g-1})$ with $$D_{\bar{F}}=(n,g_0;(a,n),(b,n),(c_1,n_1),\ldots,(c_{\ell},n_{\ell})),$$  $a+b \equiv ab \pmod{n}$, is star-realizable. 
\item Suppose that $\bar{F} = F_{\T}^m$, for some star-realizable $F_{\T}$ as in Definition~\ref{defn:starreal_FTtuple} so that the pairs $(a,n)$ (resp. $(b,n)$) belong to $D_{F_{i}}$ (resp. $D_{F_{j}}$). Then 
\begin{gather*} 
(a,n) \in \{(1, |\tilde{F}_i|)_{m_i,\tilde{F}_i}, (|\tilde{F}_i|/2-1, |\tilde{F}_i|)_{m_i,\tilde{F}_i},(|\tilde{F}_i|-2, |\tilde{F}_i|)_{m_i,\tilde{F}_i}\} \\ \text{ and } \\ (b,n) \in \{(1, |\tilde{F}_{j}|)_{{m_{j},\tilde{F}_{j}}}, (|\tilde{F}_{j}|/2-1, |\tilde{F}_{j}|)_{m_{j},\tilde{F}_{j}},(|\tilde{F}_{j}|-2, |\tilde{F}_{j}|)_{m_{j},\tilde{F}_{j}}\}
\end{gather*}
\end{enumerate}
\end{defn}

\noindent Denoting  $\eta=\frac{\mu_{z_ii}+\mu_{z_jj}-1}{n}$, we will now we give an algorithm to represent a star-realizable root as a word in Dehn twists.

\begin{algo}
\label{algo:word_root}
Consider a star-realizable root $F \in \Mod(S_g)$ of $T_c$ as in Definition~\ref{defn:star_root}. 
\begin{enumerate}[\textit{Step} 1.]
\item Apply Algorithm~\ref{algo:star_methodFT} to obtain $\W(\bar{F})$.
\item Set $$\W(F) = \W(\bar{F})(T_c)^{-\eta}.$$
\item By Lemma~\ref{lem:star}, $\W(F)$ is the desired representation of $F$ as a word in Dehn twists, up to conjugacy.
\end{enumerate}
\end{algo}
 
Let $G \in \Mod(S_g)$ be the Margalit-Schleimer root (of $T_c$) of degree $2g-1$. In~\cite{MS} an expression for $\W(G)$ was derived using the chain relation, and in~\cite{MK1} it was shown that $D_{\bar{G}} = (2g-1, 0; (2, 2g-1),(2,2g-1), (-4,2g - 1))$. In the following example, we will apply Algorithm~\ref{algo:word_root} to derive the $\W(F)$ for a root $F \in \Mod(S_g)$ of degree $2g-1$ for which $D_{\bar{F}}$ is different from $D_{\bar{G}}$ for $g \geq 3$.

\begin{exmp}
Consider a root $F \in \Mod(S_g)$ of degree $2g-1$, where $D_{\bar{F}} = (2g-1,0;(g,2g-1),(g,2g-1),(2g-2,2g-1))$. Since $\W(\bar{F}) = W_{g-1,2g-1}^{2}$, by applying Algorithm~\ref{algo:word_root}, we get $$\W(F) = T_{a_1}^{-1}(T_{c_1}T_{a_2}\prod_{i=2}^{g-1} (T_{b_i}T_{c_i})T_{b_g}T_{a_g})^2.$$
\end{exmp}

\subsubsection{Fractional roots of Dehn twists about nonseparating curves} For $g \geq 2$, a \textit{fractional root of $T_c$ of degree $(m,n)$} is an $F \in \Mod(S_g)$ such that $F^n = T_c^m$. It is known~\cite{KR2} that such a root of $T_c$ may either preserve or reverse the two sides of $c$, and a side-preserving $F$ of degree $(m,n)$ satisfies $n \leq 4g-4$. Further, a side-preserving fractional root of degree $(2g-2,4g-4)$ always exists in $\Mod(S_g)$. As in the case of roots of Dehn twists, a side-preserving fractional root $F \in \Mod(S_g)$ of degree $(m,n)$ corresponds to an $\bar{F} \in \Mod(S_{g-1})$ of order $n$ such that the $\langle \bar{\F} \rangle$-action has two distinguished fixed points where the induced rotation angles add up to $2\pi m/n \pmod{2 \pi}$. In the following result, we will assume the notation of Theorem~\ref{thm:gen_starreln}.

\begin{prop}
There exists a conjugacy class of a side-preserving fractional root of $T_{a_1}$ of degree $(2g-2,4g-4)$ represented by
$$T_{a_2}\prod_{i=1}^{g-1}(T_{c_i}T_{b_{i+1}}).$$
\end{prop}

\begin{proof}
Any root $F$ of $T_{a_1}$ of degree $(2g-2,4g-4)$ is realized from an $\bar{F} \in \Mod(S_{g-1})$ by attaching an annulus (with a $2 \pi m/n $-twist) connecting the two boundary components ($d_1$ and $d_2$) of the subsurface $S_{g-1}^2$. In particular, we consider the case when $D_{\bar{F}} = (4g-4,0;(1,2),(1,4g-4),(2g-3,4g-4))$. By Theorem~\ref{thm:gen_starreln} (for $k=2$) and Lemma~\ref{lem:per_comps}, we have
$$ (T_{a_2}\prod_{i=1}^{g-1}(T_{c_i}T_{b_{i+1}}))^{4g-4}= T_{a_1}T_{a_1}^{{(2g-3)^+}}.$$ But, as $(2g-3)^+\equiv 2g-3 \pmod {4g-4}$, this further simplifies to
$$ (T_{a_2}\prod_{i=1}^{g-1}(T_{c_i}T_{b_{i+1}}))^{4g-4}= T_{a_1}^{2g-2},$$ from which our assertion follows.
\end{proof}

\subsubsection{Dehn twists about separating curves} Let $c$ be a separating curve in $S_g$ so that $S_g = S_{g_1} \#_{c} S_{g_2}$. A pair $(F_1,F_2)$ of periodic mapping classes, where $F_i \in \Mod(S_{g_i})$, is said to be \textit{degree n root-realizing} if $\text{lcm}(|F_1|,|F_2|) = n$ and the actions of $\langle \F_i \rangle$ on the $S_{g_i}$ have distinguished fixed points $P_i \in S_{g_i}$ where the induced local rotation angles add up to $2\pi/n$ modulo $2\pi$. It was shown in~\cite{KR1} that root $F\in \Mod(S_g)$ of $T_c$ of degree $n$ corresponds to a degree $n$ root-realizing pair $(F_1,F_2)$ of periodic mapping classes. Following the theory in \cite{KR1}, we consider the degree $4g_2+2$ root realizing pair $(F_1,F_2)$, where $F_1 \in \Mod(S_{g_1})$ is the hyperelliptic involution with $D_{F_1}=(2,0;((1,2),2g_1+2))$ and $F_2 \in \Mod(S_{g_2})$ with $D_{F_2}=(2g_2+1,0;(2,2g_2+1),(2,2g_2+1),(2g_2-3,2g_2+1)).$ We remove invariant disks around the fixed points corresponding to the pairs $(1,2)$ in $D_{F_1}$ and $(2,2g_2+1)$ in $D_{F_2}$ induced by the respective actions on the $S_{g_i}$ and then attach an annulus with a $1/(4g_2+2)$-twist (connecting the resultant boundary components) to realize a root $F$ of $T_c$ in $\Mod(S_{g})$. The following proposition provides a factorization of this root into Dehn twists.

\begin{prop}\label{thm:root_sep}
For $g \geq 2$, let $c$ be separating curve in $S_g$ so that $S_g = S_{g_1}\#_cS_{g_2}$. Then there exists a root $F \in \Mod(S_g)$ of $T_c$ of degree $4g_2+2$ such that
$$\W(F)=(\prod_{j=1}^{g_1}(T_{a_j}T_{b_j}T_{a_j'})^{2(-1)^{(g_1-j)}})(T_{a_{g_1+1}}^2(\prod_{i=g_1+1}^{g-1}T_{b_i}T_{c_i})T_{b_g}T_{a_g})^{g_2+1}T_{c}^{-1}.$$ 
\end{prop}

\begin{proof}
By applying Algorithm~\ref{alg:inv}, we have
\begin{equation}
\label{eqn:1}
\left(\prod_{j=1}^{g_1}(T_{a_j}T_{b_j}T_{a_j'})^{2(-1)^{(g_1-j)}}\right)^{2(2g_2+1)}=T_{c}^{2g_2+1},
\end{equation}
by Algorithm~\ref{algo:chain_method}, we get
\begin{equation} 
\label{eqn:2}
((T_{a_{g_1+1}}^2(\prod_{i=g_1+1}^{g-1}T_{b_i}T_{c_i})T_{b_g}T_{a_g})^{(2g_2+1)})^{2(g_2+1)}=T_{c}^{2g_2+2}.
\end{equation}
From Equations~(\ref{eqn:1}) and~(\ref{eqn:2}) above, we have
	$$(\prod_{j=1}^{g_1}(T_{a_j}T_{b_j}T_{a_j'})^{2(-1)^{(g_1-j)}})^{2(2g_2+1)}(T_{a_{g_1+1}}^2(\prod_{i=g_1+1}^{g-1}T_{b_i}T_{c_i})T_{b_g}T_{a_g})^{2(2g_2+1)(g_2+1)}=T_{c}^{4g_2+3},$$ 
	from which our assertion follows.
\end{proof}

\subsubsection{Dehn twist about a multicurve} A \textit{multicurve} $\C$ in $S_g$ is a finite collection of disjoint nonisotopic essential simple closed curves in $S_g$. Let $\C = \{c_1,\ldots,c_n\}$ be a multicurve in $S_g$ then a Dehn twist $T_{\C}$ about $\C$ is defined by $T_{\C} := \prod_{i=1}^nT_{c_i}$. (This is well-defined as Dehn twists about disjoint nonisotpic curves in $S_g$ commute in $\Mod(S_g)$.) The roots of $T_{\C}$ have been classified in~\cite{KP}. We will now provide an explicit factorization of a root $F$ of $T_{\C}$ of degree $g$ when $\C = \{a_1,\ldots,a_g\}$, where $a_i$ are the curves indicated in Figure~\ref{fig:2_lick_gens}. Consider a $G \in \Mod(S_g)$ of order $g$ with $D_F=(g,1;(1,g),(g-1,g))$. Then by Algorithm~\ref{algo:surf_rotn}, we can derive an expression for $\W(G)$. From the theory in~\cite{KP}, it now follows that $(\W(G)T_{a_1})^g=\prod_{i=1}^gT_{a_i}$. Thus, we have the following proposition.
\begin{prop}
\label{prop:root_of_multicurve}
For $g \geq 2$, consider the multicurve $\C= \{a_1,\ldots,a_g\}$ in $S_g$. Then there exists a root $F$ of $T_{\C}$ of degree $g$ such that 
$$\W(F) = (\W(G)T_{a_1})^g,$$ where $D_G=(g,1;(1,g),(g-1,g))$.
\end{prop}

\noindent We end this subsection with a derivation for $\W(F)$ (from Proposition~\ref{prop:root_of_multicurve}) when $g=4$. 

\begin{exmp}
Consider the rotational mapping class $G \in \Mod(S_4)$ with $D_G = (4,1;(1,4),(3,4))$. By Algorithms~\ref{alg:inv}~and~\ref{algo:surf_rotn}, we have that $\W(G)=\W(G_1)\W(G_2)$, where 
$\W(G_1)=(T_{a_1}T_{b_1}T_{a_1'})^2(T_{a_4}T_{b_4}T_{y}T_{a_2}T_{b_2})^{-3}(T_{a_3}T_{b_3}T_{a_3'})^2$ (see Figure~\ref{fig:firstinv}), and $\W(G_2)=(T_{a_2}T_{b_2}T_{x_1}T_{a_1}T_{b_1})^{3}(T_{a_4}T_{b_4}T_{x_3}T_{a_3}T_{b_3})^{-3}$ (see Figure~\ref{fig:secondinv}) are involutions.
\begin{figure}[H]
	\centering
	\begin{subfigure}{.35\textwidth}
		\centering
		\includegraphics[width=.7\textwidth]{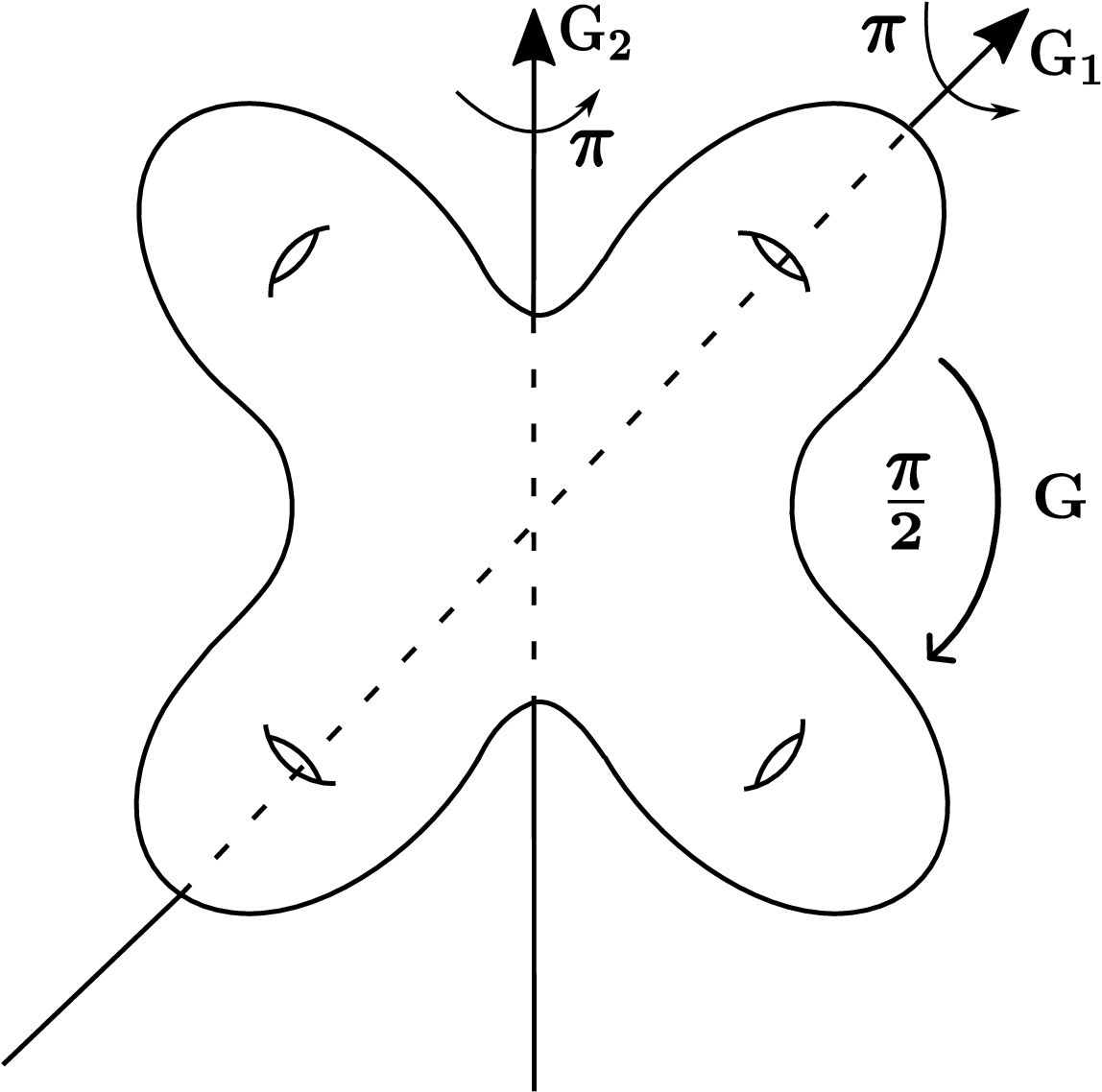}
		\caption{The periodic maps $G$, $G_1$, and $G_2$.}
		\label{fig:rootorder4}
	\end{subfigure}%
	\begin{subfigure}{.35\textwidth}
		\centering
		\includegraphics[width=\textwidth]{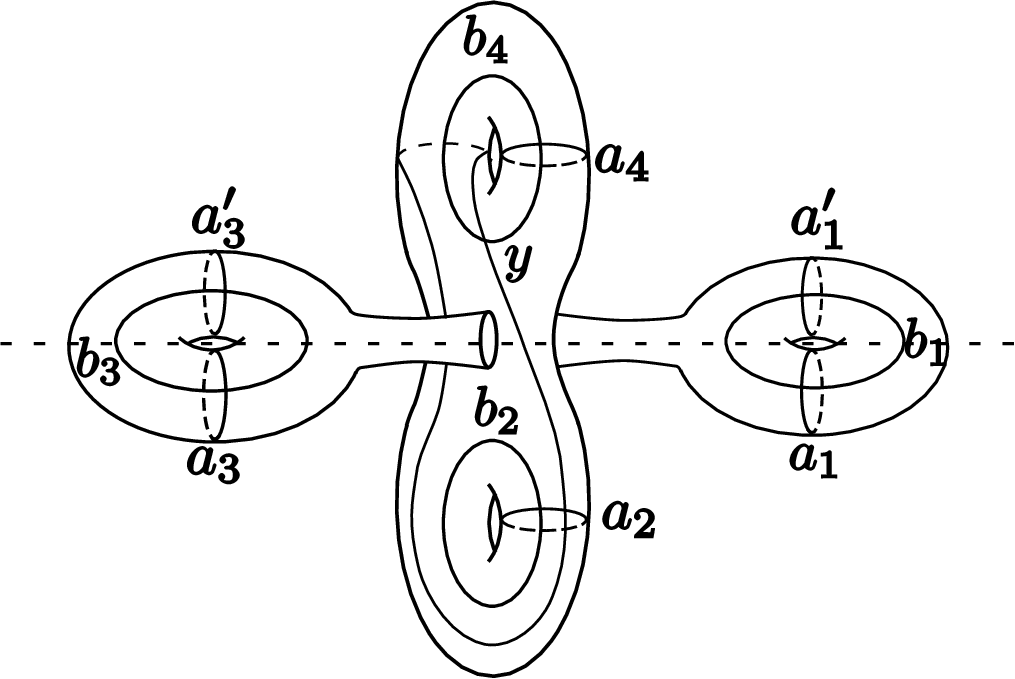}
		\caption{The involution $G_1$.}
		\label{fig:firstinv}
	\end{subfigure}%
	\begin{subfigure}{.35\textwidth}
		\centering
		\includegraphics[width=.68\textwidth]{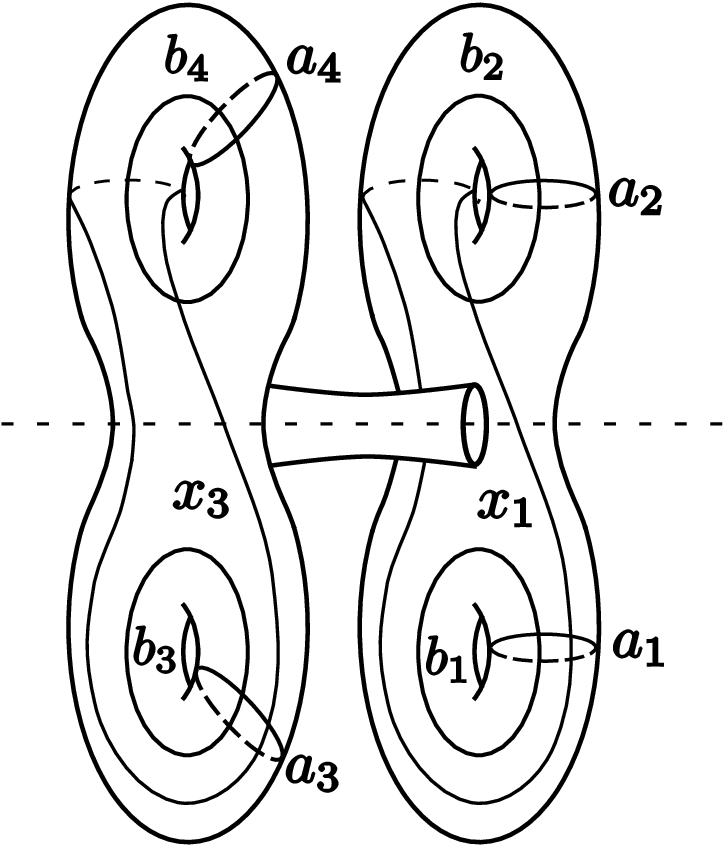}
		\caption{The involution $G_2$.}
		\label{fig:secondinv}
	\end{subfigure}
	\caption{Factorization of a rotation into a product of two involutions.}
\end{figure}

\noindent Thus, by Proposition~\ref{prop:root_of_multicurve}, we have $(\W(G)T_{a_1})^4=T_{a_1}T_{a_2}T_{a_3}T_{a_4}$.
\end{exmp}

\subsection{A construction of pseudo-Anosov mapping classes}
\label{sec:prod_irr_per}
In this subsection, we show that for $g \geq 2$, there exists conjugates of the periodic mapping classes $W_{g,4g}$ and $W_{g,4g+2}$ (from Definition~\ref{defn:star_real}) whose product is pseudo-Anosov. In this connection, we will use the following symplectic (sufficient) criterion~\cite[Proposition 2]{MS1} for a given mapping class to be pseudo-Anosov (originally due to Casson-Bleiler~\cite{CB}).
\begin{theorem}
	\label{thm:symp_crit}
	Let $F \in \Mod(S_g)$ and let $P_F(x)$ be the characteristic polynomial of $\Psi(F)$. Suppose that each of the following conditions hold. 
	\begin{enumerate}[(i)]
		\item $P_F(x)$ is symplectically irreducible over $\Z$.
		\item $P_F(x)$ is not a cyclotomic polynomial.
		\item $P_F(x)$ is not a polynomial in $x^k$ for any $k>1$.
	\end{enumerate}
	Then $F$ is pseudo-Anosov.
\end{theorem}

We now consider the conjugates 
$$W_{4g+2}'=(T_{a_1}\prod_{i=1}^{g-1}(T_{b_i}T_{c_i})T_{b_g}) \text{ and } 
W_{4g}'=(T_{b_1}^2\prod_{i=1}^{g-1}(T_{c_i}T_{b_{i+1}})T_{a_g})$$ of $W_{g,4g+2}$ and $W_{g,4g}$, respectively. (It is not hard to check that these are indeed conjugates). Let $W = W_{4g}'W_{4g+2}'$. A direct computation shows that $$P_W(x) = x^{2g}+2x^{2g-1}+3x^{2g-2}+ \cdots+ gx^{g+1}+(g+3)x^{g}+gx^{g-1}+\cdots+2x+1.$$ We will require the following technical lemma. 

\begin{lemma}
	\label{lem:root_of_unity}
	No complex root of unity can be a root of $P_W(x)$. 
\end{lemma}

\begin{proof}
	It is apparent that $x=\pm 1$ is not a root of $P_W(x).$ Suppose that an $n^\text{th}$ root of unity $e^{\iota\theta}$ for some $\theta\in \mathbb{R}$ and $n\geq 3$, is a root of $P_W(x)$. Then we have $P_W(e^{\iota\theta}) = 0$, which yields the following two equations:
	\begin{gather*}
	\sum_{j=0}^{2g}\cos(j\theta)+\sum_{j=1}^{2g-1}\cos(j\theta)+\cdots+\sum_{j=g-1}^{g+1}\cos(j\theta)+3\cos(g\theta)=0 \\
	\sum_{j=0}^{2g}\sin(j\theta)+\sum_{j=1}^{2g-1}\sin(j\theta)+\cdots+\sum_{j=g-1}^{g+1}\sin(j\theta)+3\sin(g\theta)=0
	\end{gather*}
	Applying the formulas 
	\begin{gather*}
	\sum_{j=0}^{n-1}\cos(\alpha+j\beta)=\frac{\cos(\alpha+(n-1)\beta/2)\sin (n\beta/2)}{\sin(\beta/2)} \text{ and} \\
	\sum_{j=0}^{n-1}\sin(\alpha+j\beta)=\frac{\sin(\alpha+(n-1)\beta/2)\sin (n\beta/2)}{\sin(\beta/2)}
	\end{gather*}
	in the pair of equations above, we obtain the pair of equations  
	\begin{gather*}
	\cos(g\theta)(\sin^2((g+1)\theta/2)+2\sin^2(\theta/2))=0 \\
	\sin(g\theta)(\sin^2((g+1)\theta/2)+2\sin^2(\theta/2))=0
	\end{gather*}
	These equations yield a contradiction, as there does not exist any $\theta$ such that $\cos(g\theta)=0=\sin(g\theta)$, from which our assertion follows.
\end{proof}

\noindent This leads us to the following proposition. 

\begin{prop}
	For $g \geq 2$, there exists conjugates $W_{4g}'$ and $W_{4g+2}'$ of $W_{g,4g}$ and $W_{g,4g+2}$, respectively, such that $W= W_{4g}'W_{4g+2}'$ is pseudo-Anosov in $\Mod(S_g)$.
\end{prop}

\begin{proof}
It is apparent that $P_W(x)$ satisfies condition $(iii)$ of Theorem~\ref{thm:symp_crit}. Further, condition (ii) of Theorem~\ref{thm:symp_crit} holds true in view of Lemma~\ref{lem:root_of_unity}. Finally, to show condition (i) it suffices to show that $W$ does not preserve any subsurface of $S_g$ with genus greater than $0$. To show this, we consider the chain of simple closed curves $C = \{b_1,c_1,b_2,c_2,\cdots,c_{g-1},b_g,a_g\}$ in $S_g$ (as indicated in Figure~\ref{fig:2_lick_gens}). For simplicity, we relabel the curves (appearing in sequence) in $C$ by $\{\alpha_1,\cdots,\alpha_{2g}\}$. By the properties of chain maps, we have $$W(\alpha_i)=\alpha_{i+2}, \text{ for } 1 \leq i < 2g-1$$ and further it is easily seen that there exists a $k$ such that each component of $\overline{S_g \setminus \cup_{i=0}^k W^{i}(\alpha_j)}$ has genus $0$. Consequently, it follows that $W$ cannot preserve a subsurface of $S_g$ of genus greater than $0$, from which our assertion follows.
\end{proof}

\begin{rem}
For $g \geq 2$, let $\lambda_g$ denote the homological dilatation of $W_{4g}'W_{4g+2}'$. From our computations in Mathematica 11~\cite{W1}, we know that $\lambda_3 \approx 1.81$ and it appears that $\{\lambda_g\}$ is a strictly decreasing sequence with $\lambda_{25} < 1.1$. This prompts us to conjecture that the homological dilatation of $W_{4g}'W_{4g+2}'$ converges to 1 as $g \to \infty$.
\end{rem}

\section{Acknowledgements} The authors would like to express their gratitude to the journal referee(s) for several constructive comments and suggestions that have significantly improved the exposition of the paper. The first and the second authors were partly supported during this work by the UGC-JRF and the NBHM PhD fellowships, respectively.  

\bibliographystyle{plain} 
\bibliography{periodicwords}
\end{document}